\documentclass[11pt,reqno]{amsart}
\usepackage{amssymb,amsmath}
\usepackage{color,graphicx,tikz,tikz-cd}
\usepackage{times,mathptmx}
\usepackage[utf8]{inputenc}
\usepackage[centering,includehead,width=140mm,height=21cm]{geometry}
\usepackage[colorlinks=true,linkcolor=blue]{hyperref}

\newcommand {\df}[1]{\emph {#1}}

\newcommand {\calM}{{\mathcal M}}
\newcommand {\calN}{{\mathcal N}}

\newcommand {\PP}{{\mathbb P}}
\newcommand {\QQ}{{\mathbb Q}}
\newcommand {\RR}{{\mathbb R}}
\newcommand {\ZZ}{{\mathbb Z}}

\newcommand {\rdim}{\textnormal{rdim}}
\newcommand {\vdim}{\textnormal{vdim}}

\DeclareMathOperator {\EV}{\textnormal{EV}}
\DeclareMathOperator {\BB}{\textnormal{B}}

\newcommand {\calZ}{{\mathcal Z}}

\DeclareMathOperator{\ev}{ev}

\DeclareMathOperator{\val}{val}

\DeclareMathOperator{\dist}{dist}

\DeclareMathOperator{\id}{id}

\renewcommand{\phi}{\varphi}

\newcommand{\R}{{\mathbb R}}

\newcommand{\Q}{{\mathbb Q}}

\newcommand {\dunion}{\,\mbox {\raisebox{0.25ex}{$\cdot$} \kern-1.83ex $\cup$}
  \,}

\makeatletter
\newcommand {\ie}{i.\,e.\@ifnextchar,{}{\ }}
\newcommand {\eg}{e.\,g.\@ifnextchar,{}{\ }}
\makeatother

\theoremstyle{plain}
\newtheorem{theorem}{Theorem}[section]

\newtheorem{lemma}[theorem]{Lemma}
\newtheorem{corollary}[theorem]{Corollary}
\newtheorem{conjecture}[theorem]{Conjecture}
\theoremstyle{remark}
\newtheorem{construction}[theorem]{Construction}
\newtheorem{notation}[theorem]{Notation}
\newtheorem{example}[theorem]{Example}

\newtheorem{remark}[theorem]{Remark}

\theoremstyle{definition}
\newtheorem{definition}[theorem]{Definition}

\newcommand{\mx}[2]{\tilde\calM_{0}(#1,#2)}
\newcommand{\mm}[2]{\calM_{0}(#1,#2)}
\newcommand{\mmp}[2]{\calM'_{0}(#1,#2)}

\newskip\myparskip
\renewenvironment {enumerate}%
  {\begin {oldenumerate}\parskip\myparskip \itemsep 0mm \parindent 0mm}%
  {\end {oldenumerate}}

\renewenvironment {itemize}%
  {\begin {olditemize}\parskip\myparskip \itemsep 0mm \parindent 0mm}%
  {\end {olditemize}}

\newcommand{\refx}[1]{\ref*{#1}}

\begin{document}

\title[Moduli spaces of curves in tropical varieties]%
  {Moduli spaces of curves in tropical varieties}
\author{Andreas Gathmann and Dennis Ochse}

\address{Andreas Gathmann, Fachbereich Mathematik, Technische Universität
  Kaiserslautern, Postfach 3049, 67653 Kaiserslautern, Germany}
\email{andreas@mathematik.uni-kl.de}

\address{Dennis Ochse, Fachbereich Mathematik, Technische Universität
  Kaiserslautern, Postfach 3049, 67653 Kaiserslautern, Germany}
\email{ochse@mathematik.uni-kl.de}

\thanks{\emph {2010 Mathematics Subject Classification:} 14T05, 14N35, 51M20}
\keywords{Tropical geometry, enumerative geometry, Gromov-Witten theory}

\begin{abstract}
  We describe a framework to construct tropical moduli spaces of rational
  stable maps to a smooth tropical hypersurface or curve. These moduli
  spaces will be tropical cycles of the expected dimension, corresponding to
  virtual fundamental classes in algebraic geometry. As we focus on the
  combinatorial aspect, we take the weights on certain basic $0$-dimensional
  local combinatorial curve types as input data, and give a compatibility
  condition in dimension $1$ to ensure that this input data glues to a global
  well-defined tropical cycle. As an application, we construct such moduli
  spaces for the case of lines in surfaces, and in a subsequent paper for
  stable maps to a curve \cite{GMO}.
\end{abstract}

\maketitle

\section{Introduction}

Moduli spaces of stable maps to a smooth projective variety are one of the most
important tools in modern enumerative geometry \cite{BM96,FP97}. Intersection
theory on these spaces has been used successfully to solve many enumerative
problems, such as \eg determining the numbers of plane curves of fixed genus
and degree through given points, or the numbers of rational curves of fixed
degree in a general quintic threefold \cite{Kon95,CH98}.

In recent times, tropical geometry has also been proven to be very useful for
attacking enumerative problems, starting with Mikhalkin's famous Correspondence
Theorem that provided the first link between such problems in algebraic and
tropical geometry \cite{Mi03}. Accordingly, it is an important goal in tropical
enumerative geometry to construct tropical analogues of moduli spaces of stable
maps, \ie tropical cycles whose points parametrize curves with certain
properties in a given tropical variety. This has been achieved for rational
curves in toric varieties (corresponding to tropical curves in a real vector
space) in \cite{Mik06,GKM07}, and (tropical) intersection theory on these
spaces has been used in many cases to attack and solve enumerative problems
from a purely combinatorial point of view.

Of course, it would be very desirable to have such moduli spaces of tropical
stable maps also for other target varieties. However, there is currently no
general known method to construct such spaces, mainly for the following two
reasons:

\begin{enumerate}
\item Already in algebraic geometry, the moduli spaces of stable maps may have
  bigger dimension than expected from deformation theory. One can solve this
  problem by introducing virtual fundamental classes, \ie cycle classes in the
  moduli spaces which are of the expected dimension and replace the ordinary
  fundamental classes for intersection-theoretic purposes \cite{BF97,Behrend}.
  These classes can usually be constructed as certain Chern classes of vector
  bundles. However, there is no corresponding counterpart of this theory in
  tropical geometry yet.
\item Tropical curves (in the sense of: metric graphs of the given degree and
  genus) in the given tropical variety might not be tropicalizations of actual
  algebraic curves inside the algebraic variety. Consequently, the naive
  tropical moduli space may not capture the situation from algebraic geometry
  appropriately, and it might have too big dimension even if the algebraic
  moduli space does not. This already happens for lines in cubic surfaces:
  Whereas each smooth algebraic cubic contains exactly $27$ lines, there are
  smooth tropical cubics with infinitely many lines on them \cite{Vig10}. In
  general, already this question whether a tropical curve is realizable by an
  algebraic one inside the given ambient variety is an unsolved problem. It is
  known that the space of realizable tropical curves is a polyhedral set
  \cite{Yu15}, but there is currently no explicit way to describe it.
\end{enumerate}

In this paper, we will therefore study these tropical moduli spaces from an
axiomatic and purely combinatorial point of view. Given a smooth tropical
hypersurface or curve $X$, we will describe a set of (essentially
$0$-dimensional) input data and (essentially $1$-dimensional) compatibility
conditions that allow to construct from them tropical moduli spaces $ \mm
X\Sigma $ of rational curves in $X$ that are cycles of the expected dimension
(where the subscript $0$ denotes the genus and $ \Sigma $ the degree of the
curves). We can therefore consider such cycles as tropical analogues of both
the algebraic moduli spaces of stable maps and their virtual fundamental
classes.

A central concept for this construction is the \emph{resolution dimension} $
\rdim(V) $ of a vertex $V$ of a curve in $X$. It is an integer determined by
the local combinatorial type of the curve and $X$ at $V$ that describes the
expected dimension of the moduli space of such local curves when the vertex of
the curve is resolved, modulo the local lineality space of $X$. For example, a
$4$-valent vertex in the plane $ \RR^2 $ has resolution dimension $1$, since it
can be resolved to two $3$-valent vertices in a $1$-dimensional family (modulo
translations in $ \RR^2 $). The origin of the curve in the following picture
has resolution dimension $0$ since this curve piece cannot be resolved or moved
in $X$.

\vspace{2mm plus 2mm}

\begin {center} \input {pics/intro} \end {center}

As a first rule, vertices of negative resolution dimension are not allowed in
our moduli spaces --- in the case of target curves this is known as the
Riemann-Hurwitz condition \cite{BBM10,BM13}. To construct the maximal cells of
$ \mm X\Sigma $ together with their weights, the general idea is then that
splitting and gluing of the curves allows to reduce this question to vertices
of resolution dimension $0$. For example, the picture above on the left shows a
line in a plane in $ \RR^3 $, which can vary in a $2$-dimensional moduli space
by moving its two vertices along the direction of its bounded edge. We can
split the curve in three pieces as in the picture on the right, all of which
have resolution dimension $0$. If we have weights for the moduli spaces for
these three local pieces, we can glue them back together using tropical
intersection theory to obtain a moduli space for our original situation on the
left. Technically, this means that we consider the curve pieces to have bounded
ends, and that we impose the intersection-theoretic condition that
corresponding endpoints map to the same point in $X$ by suitable evaluation
maps.

In this way, we can make $ \mm X\Sigma $ into a weighted polyhedral complex of
the expected dimension by just giving weights for vertices of resolution
dimension $0$ as input data (we will refer to them as \emph{moduli data} in
Definition \ref{def:mod-data}). However, as this input data can a priori be
arbitrary, we need a certain compatibility condition for the resulting
polyhedral complex to be balanced. A central result of our paper is that
checking this condition in resolution dimension $1$ is enough to ensure that
the gluing process then works for all dimensions of the moduli spaces
(Corollaries \ref{cor-balancing} and \ref{cor:good1dim}).

We check this condition for the moduli spaces of lines in surfaces in $ \RR^3
$, leading \eg to a well-defined $0$-dimensional moduli cycle of lines on an
arbitrary smooth tropical cubic surface, even if the actual number of such
lines is infinite. In a particular example from \cite{Vig10} of such a cubic
with infinitely many lines, we verify that this $0$-cycle still has degree $27$
as expected. In a subsequent paper, we use our methods to obtain well-defined
moduli spaces of rational stable maps (of any degree) to an arbitrary target
curve \cite{GMO}. In any case, the initial $0$-dimensional input data is
obtained by tropicalization from the algebraic situation. For example, for the
vertex in the origin in the picture above the weight is just the number of
lines in $ \PP^3 $ through $ L_1 \cap L_2 $ and $ L_3 \cap L_4 $ for any four
general lines $ L_1,L_2,L_3,L_4 \subset \PP^3 $, which is $1$.

The organization of this paper is as follows. In Chapter \ref{sec-prelim} we
give the necessary background from tropical geometry. While most of this
material is well-known, there are three techniques that go beyond the usual
theory: partially open versions of tropical varieties in Section
\ref{subsec-X}, quotient maps (and their intersection-theoretic properties) in
Section \ref{subsec:quotients}, and pull-backs of diagonals of smooth varieties
in Appendix \ref{sec:diagonal}. Chapter \ref{sec:gluing} then describes the
gluing process for curves and constructs the tropical moduli spaces from the
given input data. Finally, in Chapter \ref{sec-lines} we study the case of
lines in surfaces.

This paper is based on parts of the Ph.\,D. thesis of the second author
\cite{Och13}. It would not have been possible without extensive computations of
examples which enabled us to establish and prove conjectures about polyhedra
and their weights in our moduli spaces. For this we used the polymake extension
a-tint \cite{polymake,Ham12} and GAP \cite{gap4}. We thank an anonymous referee
for useful suggestions on the exposition. The work of the first author was
partially funded by the DFG grant GA 636/4-2, as part of the Priority Program
1489.

\section{Preliminaries} \label{sec-prelim}

\subsection{Partially open tropical varieties} \label{subsec-X}

Although most of the spaces occurring in this paper will be tropical varieties,
some of our intermediate constructions also involve ``partially open'' versions
of them. In these more general spaces the boundary faces of some polyhedra can
be missing, and thus the balancing condition is required to hold at fewer
places. The constructions in this introductory chapter are adapted to this
setting and thus sometimes slightly more general than usual. However, since all
constructions relevant to us are local, the required changes are minimal and
straightforward.

For more details on the notions of tropical cycles and fans, see \eg
\cite{GKM07,AR07}.

\begin{notation}[Polyhedra] \label{not-polyhedra}
  Let $ \Lambda $ denote a lattice isomorphic to $ \ZZ^N $ for some $ N \ge 0
  $, and let $ V := \Lambda \otimes_\ZZ \RR $ be the corresponding real vector
  space. A \emph{partially open (rational) polyhedron} in $V$ is a subset $
  \sigma \subset V $ that is the intersection of finitely many open or closed
  half-spaces $ \{ x \in V: f(x) < c \} $ resp.\ $ \{ x \in V: f(x) \le c
  \} $ for $ c \in \RR $ and $f$ in $ \Lambda^\vee $, the dual of $\Lambda$. We
  call $ \sigma $ a (closed) polyhedron if it can be written in this way with
  only closed half-spaces.

  The \emph{(relative) interior} $\sigma^\circ$ of a partially open
  polyhedron is the topological interior of $\sigma$ in its affine span. We
  denote by $ V_\sigma \subset V $ the linear space which is the shift of the
  affine span of $ \sigma $ to the origin, and set $ \Lambda_\sigma := V_\sigma
  \cap \Lambda $. The \emph{dimension} of $ \sigma $ is defined to be the
  dimension of $ V_\sigma $.

  A \emph{face} $ \tau $ of a partially open polyhedron $ \sigma $ is a
  non-empty subset of $ \sigma $ that can be obtained by changing some of the
  non-strict inequalities $ f(x) \le c $ defining $ \sigma $ into equalities.
  We write this as $ \tau \le \sigma $, or $ \tau < \sigma $ if in addition $
  \tau \neq \sigma $. If $ \dim \tau = \dim \sigma -1 $ we call $ \tau $ a
  \emph{facet} of $ \sigma $. In this case we denote by $ u_{\sigma/\tau} \in
  \Lambda_\sigma / \Lambda_\tau $ the \emph{primitive normal vector} of
  $\sigma$ relative to $\tau$, \ie the unique generator of $ \Lambda_\sigma /
  \Lambda_\tau $ lying in the half-line of $\sigma$ in $ V_\sigma / V_\tau
  \cong \RR $.
\end{notation}

\begin{definition}[Polyhedral complexes and tropical varieties]
    \label{def-complex}
  A \emph{partially open polyhedral complex} in $V$ is a collection $X$
  of partially open polyhedra in a vector space $ V = \Lambda \otimes_\ZZ \RR
  $, also called cells, such that
  \begin{enumerate}
  \item \label{def-complex:a}
    if $ \sigma \in X $ and $ \tau $ is a face of $ \sigma $ then $ \tau \in X
    $; and
  \item \label{def-complex:b}
    if $ \sigma_1,\sigma_2 \in X $ then $ \sigma_1 \cap \sigma_2 $ is empty or
    a face of both $ \sigma_1 $ and $ \sigma_2 $.
  \end{enumerate}
  It is called pure-dimensional if each inclusion-maximal cell has the same
  dimension. The \emph{support of $X$}, denoted by $|X|$, is the union of all $
  \sigma \in X $ in $V$.

  A \emph{weighted partially open polyhedral complex} is a pair $ (X,\omega_X)
  $, where $X$ is a purely $k$-dimensional partially open polyhedral complex,
  and $ \omega_X $ is a map associating a \emph{weight} $\omega_X(\sigma)\in
  \ZZ$ to each $k$-dimensional cell $ \sigma \in X $. If there is no risk of
  confusion we will write $ \omega_X $ as $ \omega $, and $ (X,\omega_X) $ just
  as $X$. A \emph{partially open tropical cycle} $X$ in $V$ is a weighted
  polyhedral complex such that for each cell $\tau$ of dimension $ k-1 $ the
  \emph{balancing condition}
    \[ \sum_{\sigma:\sigma>\tau} \omega(\sigma) \cdot u_{\sigma/\tau} = 0
      \quad \in V / V_\tau \]
  holds. It is called a \emph{partially open tropical variety} if all weights
  are non-negative. If all polyhedra in $X$ are closed, we omit the attribute
  ``partially open'' and speak \eg of tropical cycles and tropical varieties.
  A \emph{tropical fan} is a tropical variety all of whose polyhedra are cones.

  Often, the exact polyhedral complex structure of our cycles is not important.
  We call two (partially open) cycles equivalent if they allow a common
  refinement (where a refinement is required to respect the weights, and every
  polyhedron of a refinement must be closed in its corresponding cell of the
  original cycle). By abuse of notation, the corresponding equivalence classes
  will again be called (partially open) tropical cycles.

  A \emph{morphism} $ f : X \to Y $ between (partially open) tropical cycles
  $X$ and $Y$ is a locally affine linear map $ f: |X| \to |Y| $, with
  the linear part induced by a map between the underlying lattices. It is
  called an \emph{isomorphism} if it has a two-sided inverse (up to
  refinements) and respects the weights.
\end{definition}

\begin{example}[Linear spaces] \label{ex-linear}
  Let $V=\RR^k$, denote by $e_i$ for $i=1,\ldots,k$ the negatives of the
  vectors of the standard basis, and set $e_0=-e_1-\cdots-e_k$. For $ r<k $ we
  denote by $ L^k_r $ the tropical fan whose simplicial cones are indexed by
  subsets $ I \subset \{0,\ldots,k\}$ with at most $r$ elements and given by
  the cones generated by all $ e_i $ with $ i \in I $. The weights of the
  top-dimensional cones, corresponding to subsets of size $r$, are all set to
  $1$. This is the tropicalization of a general $r$-dimensional linear space
  over the Puiseux series with constant coefficient equations
  \cite[Proposition 2.5 and Theorem 4.1]{FS05}.

  The following pictures illustrate these spaces, where all displayed cones are
  thought to be extending to infinity. If instead we interpret the pictures as
  bounded spaces they represent partially open tropical varieties obtained by
  intersecting $ L_r^k $ with an open bounded polyhedron.

  \begin {center} \input {pics/example} \end {center}
\end{example}

\begin{remark} \label{rem-complex}
  In our Definition \ref{def-complex} it is allowed that two partially open
  polyhedra in a complex do not intersect although their closures do. For
  example, in the pictures above we could replace all polyhedra by their
  relative interiors. This would give us weighted partially open polyhedral
  complexes with the same support, and whose face relations and balancing
  conditions are trivial. However, spaces of this type will not occur in our
  constructions in this paper --- we will always have partially open polyhedral
  complexes $X$ such that $ \sigma \cap \tau = \emptyset $ for given $ \sigma,
  \tau \in X $ implies $ \overline \sigma \cap \overline \tau = \emptyset $.
\end{remark}

\begin{definition}[Smooth varieties] \label{def-smooth}
  For simplicity, in this paper we will follow \cite{All09} and call a
  tropical variety $X$ \emph{smooth} if it is locally isomorphic to some $
  L^k_r \times \RR^m $ around $0$ at each point (where $ r+m = \dim X $ is
  fixed, but otherwise $ k,r,m $ may depend on the chosen point). This is more
  special than the usual definition of smoothness which allows any polyhedral
  complex locally isomorphic to a matroid fan as in Appendix
  \ref{sec:diagonal}. We expect that our results would hold in this more
  general setting as well.

  For a smooth variety $X$, following \cite[Section 5.3]{Mik06} the
  \emph{canonical divisor} $ K_X $ of $X$ is defined to be the weighted
  polyhedral complex given by the codimension-$1$ skeleton of $X$, where the
  weight of a codimension-$1$ cell of $X$ is the number of adjacent maximal
  cells minus $2$.
\end{definition}

\begin{example}[Smooth curves and hypersurfaces] \label{ex-smooth}
  In this paper, the following two cases of smooth varieties will be of
  particular importance.
  \begin{enumerate}
  \item \label{ex-smooth:a}
    Consider a connected $1$-dimensional tropical variety $X$ in $\RR^N$ which
    is a tree --- we will refer to such a space as a \emph{rational curve}. The
    smoothness condition then means that for any vertex of $X$ the primitive
    integer vectors in the directions of the adjacent edges can be mapped to $
    e_0,\ldots,e_k$ in $ \RR^k $ for some $ k \leq N $ by a $\ZZ$-linear map,
    and that all edges have weight $1$. Such a vertex occurs in the canonical
    divisor $ K_X $ with weight $ k-1 $.
  \item \label{ex-smooth:b}
    Consider a hypersurface $X$ in $\RR^N$, \ie a tropical variety given by a
    \emph{tropical polynomial} in $N$ variables \cite{RST03}. The coefficients
    of this polynomial determine a subdivision of its Newton polytope, and the
    weighted polyhedral complex structure of $X$ is induced by this
    subdivision. Smoothness then means that this subdivision is unimodular
    \cite[Proposition 3.11]{Mi03}, and hence $ K_X $ contains each
    codimension-$1$ cell of $X$ with weight $1$. Most important for us will be
    the case of a cubic surface in $ \RR^3 $, \ie of $X$ being dual to a
    unimodular subdivision of the lattice polytope in $\ZZ^3$ with vertices
    $(0,0,0)$, $(3,0,0)$, $(0,3,0)$, and $(0,0,3)$. 
  \end{enumerate}

  \begin {center} \input {pics/smooth} \end {center}
\end{example}

\subsection{Tropical quotients} \label{subsec:quotients}

In this section we will define quotients of partially open tropical cycles by
vector spaces in certain cases.

\begin{definition}[Lineality space] \label{def-lineality}
  Let $X$ be a partially open polyhedral complex in a vector space $ V =
  \Lambda \otimes_\ZZ \RR $, and let $ L \subset V $ be a vector subspace
  defined over $ \QQ $. We say that $L$ is a \emph{lineality space} for $X$ if,
  for a suitable complex structure of $X$, for all $ \sigma \in X $ and $ x \in
  \sigma $ the intersection $ \sigma \cap (x+L) $ is open in $ x+L $ and equal
  to $ |X| \cap (x+L) $.
\end{definition}

\begin{remark} \label{rem-lineality}
  If $X$ is a tropical variety, \ie $ \sigma $ is closed in $V$ for all $
  \sigma \in X $, then $ \sigma \cap (x+L) $ can only be open and non-empty
  if it is all of $ x+L $. So in this case we arrive at the usual notion of
  lineality space found in the literature (although note that most authors
  only call a maximal subspace with this property a lineality space).
\end{remark}

\begin{lemma} \label{lem-lineality}
  Let $X$ be a partially open polyhedral complex in $V$ with a lineality space
  $L$. Denote by $ q: V \to V/L $ the quotient map, where $ V/L $ is considered
  to have the underlying lattice $ \Lambda/(\Lambda \cap L) $. Then for all $
  \sigma, \tau \in X $ we have:
  \begin{enumerate}
  \item \label{lem-lineality:a}
    $ q(\sigma) $ is a partially open polyhedron of dimension $ \dim q(\sigma)
    = \dim \sigma - \dim L $.
  \item \label{lem-lineality:b}
    If $ \tau \le \sigma $ then $ q(\tau) \le q(\sigma) $.
  \item \label{lem-lineality:c}
    $ q(\sigma \cap \tau) = q(\sigma) \cap q(\tau) $.
  \item \label{lem-lineality:d}
    If $ q(\sigma) = q(\tau) $ then $ \sigma=\tau $.
  \item \label{lem-lineality:e}
    $ \Lambda_{q(\sigma)} = \Lambda_\sigma / (\Lambda \cap L) $; and if
    $ \tau $ is a facet of $ \sigma $ then $ u_{q(\sigma)/q(\tau)} =
    \overline{u_{\sigma/\tau}} $ with this identification.
  \end{enumerate}
\end{lemma}

\begin{proof}
  By induction it suffices to prove the statements for $ \dim L = 1 $.
  We choose coordinates $ (x,y) \in \RR^{\dim V-1} \times \RR \cong V $ such
  that $ L = \{ (x,y): x=0 \} $, and consider $x$ as coordinates on $ V/L $.

  \refx{lem-lineality:a} In the defining inequalities for $ \sigma $ we may
  assume that all of them that contain $y$ are strict: if one of the non-strict
  defining inequalities $ f(x,y) \le c $ containing $y$ was satisfied as an
  equality at a point $ (x_0,y_0) \in \sigma $ it could not be satisfied in a
  neighborhood of $ y_0 \in \RR $, contradicting the openness of $ \sigma \cap
  ((x_0,y_0)+L) $. Hence $ \sigma $ can be written as
    \[ \sigma = q^{-1}(\sigma_0) \cap
         \{ (x,y): y > f_i(x)+a_i \text{ and } y < g_j(x)+b_j
         \text{ for all $i,j$} \} \tag {1} \]
  for some linear forms $ f_i, g_j $, constants $ a_i,b_j $, and a partially
  open polyhedron $ \sigma_0 $ given by the defining inequalities of $ \sigma $
  that do not contain $y$. But then
    \[ q(\sigma) = \sigma_0 \cap \{ x: f_i(x)+a_i < g_j(x)+b_j
         \text{ for all $ i,j $} \}, \tag {2} \]
  since for $x$ satisfying these conditions we can always find $ y \in \RR $
  with $ f_i(x)+a_i < y < g_j(x)+b_j $ for all $ i,j $. Hence $ q(\sigma) $ is
  a partially open polyhedron. Moreover, the openness of $ \sigma \cap
  ((x,y)+L) $ means that all non-empty fibers of $ q|_\sigma $ have dimension $
  \dim L $, so that $ \dim q(\sigma) = \dim \sigma - \dim L $.

  \refx{lem-lineality:b} If $ \tau \le \sigma $ is obtained from $ \sigma $ by
  changing some non-strict inequalities to equalities this means that $ \tau $
  can be written as in $(1)$ for some $ \tau_0 \le \sigma_0 $ and with the same
  $ f_i,g_j $. Then $(2)$ holds for $ \tau $ and $ \tau_0 $ as well, and we
  conclude that $ q(\tau) $ is a face of $ q(\sigma) $.

  \refx{lem-lineality:c} The inclusion ``$ \subset $'' is obvious. Conversely,
  if $ x \in q(\sigma) \cap q(\tau) $ then there are $ y,y' \in \RR $ with
  $ (x,y) \in \sigma $ and $ (x,y') \in \tau $. Hence
    \[ (x,y) \in \sigma \cap ((x,y')+L) \subset |X| \cap ((x,y')+L), \]
  which implies $ (x,y) \in \tau \cap ((x,y')+L) $ by definition of a lineality
  space. Hence $ (x,y) \in \sigma \cap \tau $, \ie $ x = q(x,y) \in q(\sigma
  \cap \tau) $.

  \refx{lem-lineality:d} By \refx{lem-lineality:c}, the equality $ q(\sigma) =
  q(\tau) $ implies $ q(\sigma \cap \tau) = q(\sigma) $. In particular, $
  \sigma \cap \tau \neq \emptyset $. Hence $ \sigma \cap \tau $ is a face of $
  \sigma $, which by the dimension statement of \refx{lem-lineality:a} must be
  of the same dimension as $ \sigma $. But this means that $ \sigma \cap \tau =
  \sigma $, and by symmetry thus also $ \sigma \cap \tau = \tau $.

  \refx{lem-lineality:e} From the definition of a lineality space it follows
  that $ L \subset V_\sigma $; hence $ V_{q(\sigma)} = V_\sigma/L $ and thus
  $ \Lambda_{q(\sigma)} = \Lambda_\sigma / (\Lambda \cap L) $. In particular,
  for $ \tau $ a facet of $ \sigma $ we have $ \Lambda_\sigma / \Lambda_\tau =
  \Lambda_{q(\sigma)} / \Lambda_{q(\tau)} $ (with the isomorphism given by
  taking quotients by $ \Lambda \cap L $), from which the statement about the
  primitive normal vector follows.
\end{proof}

\begin{corollary}[Quotients] \label{cor-quotients}
  Let $X$ be an $n$-dimensional partially open tropical variety in $ V =
  \Lambda \otimes_\ZZ \RR $ with a lineality space $L$, and let $ q: V \to V/L
  $ be the quotient morphism. Then
    \[ X/L := \{ q(\sigma): \sigma \in X \} \]
  together with the weights $ \omega_{X/L}(q(\sigma)) := \omega_X(\sigma) $
  is a partially open tropical variety of dimension $ n-\dim L $ in the vector
  space $ V/L $ with lattice $ \Lambda / (\Lambda \cap L) $. We will also
  denote it by $ q(X) $.
\end{corollary}

\begin{proof}
  By Lemma \ref{lem-lineality} \refx{lem-lineality:a} we see that $ X/L $ is a
  collection of partially open polyhedra which by \refx{lem-lineality:d} are in
  bijection to the polyhedra in $X$. The statements \refx{lem-lineality:b} and
  \refx{lem-lineality:c} of the lemma now imply that $ X/L $ is a partially
  open polyhedral complex as in Definition \ref{def-complex}. Moreover, the
  dimension statement of part \refx{lem-lineality:a} of the lemma means that $
  X/L $ is of pure dimension $ n - \dim L $, and that $ \omega_{X/L}(q(\sigma))
  := \omega_X(\sigma) $ defines a weight function on the top-dimensional cones.
  Finally, part \refx{lem-lineality:e} of the lemma shows that the images of
  the balancing conditions for $X$ in $V$ give us the balancing conditions for
  $ X/L $ in $ V/L $.
\end{proof}

Our main examples for this quotient construction are (cycles in) the moduli
spaces of tropical curves, which we introduce in Section \ref{subsec-m0nrd}.

\subsection{Tropical intersection theory} \label{subsec-tropint}

We now brief\/ly recall some constructions and results of tropical intersection
theory. For a detailed introduction we refer to \cite{AR07,ShaPhD,Sha10}.
Although the theory is only developed for closed tropical cycles there, the
extension to the case of partially open tropical cycles stated below follows
immediately since all constructions involved are local.

\begin{construction}[Rational functions and divisors]
    \label{constr-divisor}
  For a partially open pure-dimensional tropical cycle $X$, a \emph{(non-zero)
  rational function on $X$} is a continuous function $ \varphi : |X| \to \RR $
  that is affine linear on each cell (for a suitable polyhedral complex
  structure), with linear part given by an element of $ \Lambda^\vee $. We can
  associate to such a rational function $\varphi$ a \emph{(Weil) divisor},
  denoted by $\varphi \cdot X$. It is a partially open tropical subcycle of $X$
  (i.e.\ a partially open cycle whose support is a subset of $ |X| $) of
  codimension $1$. Its support is contained in the subset of $|X|$ where
  $\varphi$ is not locally affine linear \cite[Construction 3.3]{AR07}.

  Two important examples for rational functions and divisors in this paper are:
  \begin{enumerate}
  \item \label{constr-divisor:a}
    On $X=\RR^N$, a tropical polynomial $\varphi$ defines a rational function,
    and its divisor is just the tropical hypersurface defined by $\varphi$.
  \item \label{constr-divisor:b}
    If $X$ is one-dimensional, the divisor $ \varphi \cdot X $ consists of
    finitely many points. We define its \emph{degree} of to be the sum of
    the weights of the points in $\varphi \cdot X$. By abuse of notation, we
    sometimes write this degree as $\varphi \cdot X$ as well.
  \end{enumerate}
  Multiple intersection products $\varphi_1 \cdot \; \cdots \; \cdot \varphi_r
  \cdot X$ are commutative by \cite[Proposition 3.7]{AR07}.
\end{construction}

\begin{remark}[Pull-backs and push-forwards] \label{rem-push-pull}
  Rational functions on a (partially open) tropical cycle $Y\subset \Lambda'
  \otimes_\ZZ \RR$ can be \emph{pulled back} along a morphism $f:X\rightarrow
  Y$ to rational functions $ f^* \varphi = \varphi \circ f$ on $X$. Also, we
  can \emph{push forward} subcycles $Z$ of $X$ to subcycles $ f_* Z $ of $Y$ of
  the same dimension \cite[Proposition 4.6 and Corollary 7.4]{AR07}, where in
  the partially open case we will always restrict ourselves to injective maps
  $f$ so that no problems can arise from two partially open polyhedra with
  different boundary behavior that are mapped by $f$ to an overlapping image.
  In any case, by picking a suitable refinement of $X$ we can ensure that the
  partially open image polyhedra $f(\sigma)$ for $ \sigma \in X $ form a
  partially open polyhedral complex $ f_* Z $. For a top-dimensional cell $
  \sigma'\in f_* Z $ its weight $ \omega_{f_* Z} (\sigma')$ is given by
    \[ \omega_{f_* Z} (\sigma') :=
         \sum_{ \sigma}
	 \omega_X (\sigma) \cdot |\Lambda'_{\sigma'}/f(\Lambda_\sigma)|, \]
  where the sum goes over all top-dimensional cells $\sigma\in Z$ with
  $f(\sigma) = \sigma'$. Of course, in the partially open case, there will be
  at most one such $ \sigma $ in each sum since $f$ is assumed to be injective.
  As expected, push-forwards and pull-backs satisfy the projection formula
  \cite[Proposition 4.8 and Corollary 7.7]{AR07}.
\end{remark}

\begin{construction}[Pull-backs along quotient maps] \label{constr-pull-quot}
  Let $X$ be a partially open tropical variety with a lineality space $L$, so
  that there is a quotient variety $ X/L $ as in Corollary \ref{cor-quotients}
  with quotient map $ q: X \to X/L $. Moreover, let $Z$ be a partially open
  subcycle of $ X/L $. Then the collection of polyhedra $ q^{-1}(\sigma) $ for
  $ \sigma \in Z $, together with the weight function $ \omega(q^{-1}(\sigma))
  = \omega_Z(\sigma) $, is a partially open subcycle of $X$ of dimension $ \dim
  Z + \dim L $ (in fact, the balancing conditions follows from Lemma
  \ref{lem-lineality} \refx{lem-lineality:e}). We denote it by $ q^* Z $.
\end{construction}

We conclude this short excursion into tropical intersection theory with two
compatibility statements between the quotient construction and pull-backs of
rational functions resp.\ push-forward of cycles.

\begin{lemma} \label{lem:quot-pull}
  Let $X$ be a partially open tropical cycle with lineality space $L$ and
  quotient map $ q: X \to X/L $. Then for any rational function $ \phi $ on $
  X/L $ we have
    \[ (q^* \phi \cdot X) / L = \phi \cdot (X/L). \]
\end{lemma}

\begin{proof}
  This is obvious from the definitions.
\end{proof}

\begin{lemma} \label{lem:quot-push}
  Let $ f: X \to Y $ be a morphism between partially open tropical cycles.
  Assume that $X$ and $Y$ have lineality spaces $L$ and $ L' $, respectively,
  and that the linear part of $f$ on each cell maps $ \Lambda_X \cap L $
  isomorphically to $ \Lambda_Y \cap L' $. If $ g : X/L \to Y/L' $ is the
  morphism giving a commutative diagram
    \[ \begin{tikzcd}
         X \arrow{r}{f} \arrow{d}{q} & Y \arrow{d}{q'} \\
         X/L \arrow{r}{g} & Y/L'
       \end{tikzcd} \]
  then $ f_* (X)/L' = g_*(X/L). $
\end{lemma}

\begin{proof}
  Assume that the polyhedral complex structure of all cycles is sufficiently
  fine to be compatible with all of the morphisms. Let $ \sigma $ be a maximal
  cell in $X$, and set $ \tau = f(\sigma) $ and $ \rho = q'(\tau) $. Applying a
  suitable translation, we may assume that $f$ (and thus also $g$) is linear
  on $ \sigma $.
  
  If $ \dim \rho < \dim \sigma - \dim L $ then $ \sigma $ does not
  contribute to either cycle in the statement of the lemma. Otherwise, the
  assumption implies that $f$ maps $ \Lambda_X \cap L $, and hence every
  saturated sublattice of $ \Lambda_X \cap L $ such as $ \Lambda_\sigma \cap L
  $, to a saturated sublattice of $ \Lambda_Y $. Hence, in the inclusion
    \[ f(\Lambda_\sigma \cap L) \subset f(\Lambda_\sigma) \cap f(L)
       = f(\Lambda_\sigma) \cap L' \subset \Lambda_\tau \cap L' \]
  we must have equality since both sides are saturated lattices in $ \Lambda_Y
  $ of the same rank. By the last equality in this chain it then follows from
  the weight formulas of Corollary \ref{cor-quotients} and Remark
  \ref{rem-push-pull} that the contribution of $ \sigma $ to the cell $ \rho $
  in the two cycles of the lemma is
  \begin{align*}
    \omega_{f_*(X)/L'} (\rho)
    &= \omega_X (\sigma) \cdot |\Lambda_\tau / f(\Lambda_\sigma)| \\
    &= \omega_X (\sigma) \cdot |(\Lambda_\tau/(\Lambda_\tau \cap L')) \,/\,
       (f(\Lambda_\sigma)/(f(\Lambda_\sigma) \cap L'))| \\
    &= \omega_X (\sigma) \cdot |q'(\Lambda_\tau) / q'(f(\Lambda_\sigma)) | \\
    &= \omega_X (\sigma) \cdot |q'(\Lambda_\tau) / g(q(\Lambda_\sigma)) | \\
    &= \omega_{g_*(X/L)}(\rho).
  \end{align*}
\end{proof}

\subsection{Tropical moduli spaces of curves} \label{subsec-m0nrd}

We will now come to the construction of the moduli spaces $ \mm {\RR^N}{\Sigma}
$ of tropical curves in $ \RR^N $. Analogously to the algebraic case, elements
of these spaces will be given by a rational $n$-marked abstract tropical curve,
together with a map to $ \RR^N $ whose image is a (not necessarily smooth)
tropical curve in $ \RR^N $ as in Section \ref{subsec-X}, with unbounded
directions as determined by $ \Sigma $.

For later purposes we will also need a version $ \mmp {\RR^N}{\Sigma} $ of
these spaces where the ends $ x_i $ have bounded lengths. In contrast to the
original spaces $ \mm {\RR^N}{\Sigma} $ they will only be partially open
tropical varieties since there is no limit curve when the length of such a
bounded end approaches zero.

Let us start with the discussion of abstract tropical curves.

\begin{construction}[Abstract curves] \label{constr-m0n}
  A \emph{(rational) abstract tropical curve} \cite[Definition 3.2]{GKM07} is a
  metric tree graph $ \Gamma $ with all vertices of valence at least $3$.
  Unbounded ends (with no vertex there) are allowed and labeled by $ x_1,\dots,
  x_n $. The tuple $ (\Gamma,x_1,\dots,x_n) $ will be referred to as an
  \emph{$n$-marked abstract tropical curve}. Two such curves are called
  isomorphic (and will from now on be identified) if there is an isometry
  between them that respects the markings. The set of all $n$-marked abstract
  tropical curves (modulo isomorphisms) is denoted $ \calM_{0,n} $.

  For $ n \ge 3 $ it follows from \cite[Theorem 3.4]{SS04a},
  \cite[Section 2]{Mi07}, or \cite[Theorem 3.7]{GKM07} that $\calM_{0,n}$ can
  be given the structure of a tropical fan as follows: Consider the map
    \[ d: \calM_{0,n} \to \RR^{\binom{n}{2}}, \;\;
          (\Gamma,x_1,\dots,x_n) \mapsto (\dist_{\Gamma}(x_i,x_j))_{i<j} \]
  where $\dist_{\Gamma}(x_i,x_j)$ denotes the distance between the two marked
  ends $x_i$ and $x_j$ in the metric graph $\Gamma$. Moreover, consider the
  linear map
    \[ \phi: \RR^n \to \RR^{\binom n2}, \;\;
       (a_i)_i \mapsto (a_i+a_j)_{i,j}, \]
  let $ u_i = \phi(e_i) $ be the images of the unit vectors, let $ U_n :=
  \phi(\RR^n) = \langle u_1,\dots,u_n \rangle $, and
    \[ q_n: \RR^{\binom n2} \to Q_n := \RR^{\binom n2} / U_n \]
  be the quotient map. For a subset $I\subset\left[n\right]$ with $2\le |I| \le
  n-2 $ define $v_I$ to be the image under $ q_n \circ d: \calM_{0,n} \to Q_n $
  of a tree $\Gamma$ with exactly one bounded edge of length one, the marked
  leaves $x_i$ with $i\in I$ on one side and the leaves $x_i$ for $i\notin I$
  on the other. Let $ \Lambda_n :=\langle v_I \rangle_\ZZ\subset Q_n$ be the
  lattice in $Q_n$ generated by the vectors $v_I$. By \cite[Theorem 4.2]{SS04a}
  the map $ q_n \circ d:\calM_{0,n} \to Q_n$ is injective, and its image is a
  purely $ (n-3) $-dimensional simplicial tropical fan in $ Q_n = \Lambda_n
  \otimes_\ZZ \RR $, with all top-dimensional cones having weight $1$. In the
  following we will always consider $ \calM_{0,n} $ with this structure of a
  tropical fan.

  The cones of $ \calM_{0,n} $ are labeled by the \emph{combinatorial types} of
  marked curves, \ie by the homeomorphism classes of the curves relative to
  their ends. The vectors $v_I$ generate the rays of $\calM_{0,n}$.
\end{construction}

\begin{construction}[Abstract curves with bounded ends] \label{constr-m0nb}
  We will now adapt Construction \ref{constr-m0n} to the case when the ends
  also have bounded lengths. So we say that an \emph{$n$-marked abstract
  tropical curve with bounded ends} is a metric tree graph as above with the
  unbounded ends replaced by bounded intervals, \ie a metric graph $ \Gamma $
  without $2$-valent vertices, and the $1$-valent vertices labeled by $
  x_1,\dots,x_n $. The notions of isomorphisms and combinatorial types carry
  over from the case with unbounded ends. The set of all $n$-marked curves
  (modulo isomorphisms) with bounded ends is denoted $ \calM'_{0,n} $.

  To make $ \calM'_{0,n} $ for $ n \ge 3 $ into a partially open tropical
  variety we consider the distance map $ d: \calM'_{0,n} \to \RR^{\binom n2} $
  as above, which in this case however includes the lengths of the bounded
  ends. Then $d$ is injective: the vectors $ u_i = \phi(e_i) \in Q_n $ in the
  notation of Construction \ref{constr-m0n} correspond exactly to a change of
  the length of the bounded edge at $ x_i $ --- so by Construction
  \ref{constr-m0n} the image point under $ q_n \circ d: \calM'_{0,n} \to Q_n $
  allows to reconstruct the combinatorial type of the graph as well as the
  lengths of all edges not adjacent to the markings, whereas the full vector in
  $ \RR^{\binom n2} $ then allows to reconstruct the lengths of the ends as
  well.

  Note that the combinatorial types of these curves with bounded ends are in
  one-to-one correspondence with the types of curves with unbounded ends, and
  that $d$ maps $ \calM'_{0,n} $ to a partially open polyhedral complex (with
  cones in bijection to the combinatorial types), which by abuse of notation we
  will also denote by $ \calM'_{0,n} $. Taking again the lattice in $
  \RR^{\binom n2} $ generated by all graphs with integer lengths, and giving
  all cells of $ \calM'_{0,n} $ weight $1$, we get in fact a partially open
  tropical variety in $ \RR^{\binom n2} $ of dimension $ 2n-3 $. It has
  lineality space $ Q_n $ in the sense of Definition \ref{def-lineality}, and
  we have $ \calM'_{0,n} / Q_n = \calM_{0,n} $ as in Corollary
  \ref{cor-quotients}.
\end{construction}

\begin{construction}[Parametrized curves in $ \RR^N $]
    \label{constr-curve}
  A \df{(rational, parametrized) curve in} $ \RR^N $
  \cite[Definition 4.1]{GKM07} is a tuple $ (\Gamma,x_1,\dots,x_n,h) $, with $
  (\Gamma,x_1,\dots,x_n) $ a rational $n$-marked abstract tropical curve and $
  h: \Gamma \to \RR^N $ a continuous map satisfying:
  \begin{enumerate}
  \item On each edge $e$ of $ \Gamma $, with metric coordinate $t$, the map $h$
    is of the form $ h(t) = a + t \cdot v $ for some $ a \in \RR^N $ and $ v
    \in \ZZ^N $. If $ V \in e $ is a vertex and we choose $t$ positive on $e$,
    the vector $v$ will be denoted $v(e,V)$ and called the \df{direction} of
    $e$ (at $V$). If $e$ is an end and $t$ pointing in its direction, we write
    $v$ as $v(e)$.
  \item For every vertex $V$ of $ \Gamma $ we have the \df{balancing condition}
      \[ \sum_{e: V \in e} v(e,V) = 0. \]
  \end{enumerate}
  Two such curves in $ \RR^N $ are called isomorphic (and will be identified)
  if there is an isomorphism of the underlying abstract $n$-marked curves
  commuting with the maps to $ \RR^N $.

  The \df{degree} of a curve $ (\Gamma,x_1,\dots,x_n,h) $ in $ \RR^N $ as above
  is the $n$-tuple $ \Sigma = (v(x_1),\dots,v(x_n)) \in (\ZZ^N)^n $ of
  directions of its ends. We denote the space of all curves in $ \RR^N $ of a
  given degree $\Sigma$ by $ \mm {\RR^N}{\Sigma} $.

  Note that $ \Sigma $ may contain zero vectors, corresponding to contracted
  ends that can be thought of as marked points in the algebraic setting (see
  Construction \ref{constr-ev}).

  If $ n \ge 3 $ and there is at least one end $ x_i $ with $ v(x_i)=0 $ we can
  use the bijection
    \[ \mm {\RR^N}{\Sigma} \to \calM_{0,n} \times \RR^N, \quad
       (\Gamma,x_1,\dots,x_n,h) \mapsto ((\Gamma,x_1,\dots,x_n), h(x_i)) \]
  (which forgets $h$ except for its image on $ x_i $) to give it the structure
  of a tropical variety \cite[Proposition 4.7]{GKM07}.
\end{construction}

\begin{construction}[Evaluation maps] \label{constr-ev}
  For each $i$ with $ v(x_i)=0 $ there is an \emph{evaluation map} 
    \[ \ev_i : \mm {\RR^N}{\Sigma} \rightarrow \RR^N \]
  assigning to a tropical curve $(\Gamma,x_1,\dots,x_n,h)$ the position
  $h(x_i)$ of its $i$-th marked end (note that this is well-defined since the
  marked end $x_i$ is contracted to a point). By \cite[Proposition 4.8]{GKM07},
  these maps are morphisms of tropical fans.
\end{construction}

\begin{construction}[Parametrized curves in $ \RR^N $ with bounded ends]
    \label{constr-m0nird}
  Constructions \ref{constr-m0nb} and \ref{constr-curve} can obviously be
  combined to obtain moduli spaces $ \mmp {\RR^N}{\Sigma} $ of curves in $
  \RR^N $ with bounded ends, as pull-backs under the quotient maps that forget
  the lengths of the bounded ends. They are partially open tropical varieties
  and admit evaluation maps to $ \RR^N $ as in Construction \ref{constr-ev} at
  all ends (\ie not just at the contracted ones).
\end{construction}

\begin{remark}[Parametrized curves in $ \RR^N $ with few markings]
    \label{rem-m0nird}
  In the following, we will also need the moduli spaces $ \mm {\RR^N}{\Sigma} $
  of parametrized curves $ (\Gamma,x_1,\dots,x_n,h) $ for the special cases
  when $ n<3 $ or there is no $ x_i $ with $ v(x_i) = 0 $, so that the above
  bijection with $ \calM_{0,n} \times \RR^N $ is not available. In order to
  overcome this technical problem and still give $ \mm {\RR^N}{\Sigma} $ the
  structure of a tropical variety, there are two possibilities:
  \begin{enumerate}
  \item \label{rem-m0nird:a}
    One can use \emph{barycentric coordinates}, taking a certain weighted
    average of the vertices of the curve.
  \item \label{rem-m0nird:b}
    One can combine evaluation maps at several non-contracted ends, where at
    each such end the evaluations are only taken modulo the direction of the
    edge to make them well-defined.
  \end{enumerate}
  Details on these alternative construction can be found in
  \cite[Section 1.2]{Och13}. In the following, we will just assume that $ \mm
  {\RR^N}{\Sigma} $ has the structure of a tropical variety in any case.
\end{remark}

\begin{remark}[Degree of the canonical divisor on curves]
    \label{rem-degree-kx}
  As in Example \ref{ex-smooth}, let $ X \subset \RR^N $ be a smooth rational
  curve or a smooth hypersurface. Consider an $n$-marked curve $
  (\Gamma,x_1,\dots,x_n,h) \in \mm {\RR^N}{\Sigma} $ with $ h(\Gamma) \subset X
  $, so that $h$ can also be viewed as a morphism from the (abstract) tropical
  curve $ \Gamma $ to $X$.

  We can then consider $ K_X $ as a divisor on $X$ as in Construction
  \ref{constr-divisor} and compute its pull-back $ h^* K_X $ according to
  Remark \ref{rem-push-pull}. Its degree (as in Construction
  \ref{constr-divisor} \refx{constr-divisor:b}) depends only on $X$ and $
  \Sigma $, and not on $h$:
  \begin{enumerate} \parindent 0mm \parskip 1.3ex plus 0.5ex minus 0.3ex
  \item \label{rem-degree-kx:a}
    If $X$ is a rational curve note that any two points in $X$ are
    rationally equivalent divisors in the sense of \cite{AR07}. Hence the
    degree of the divisor $ h^* P $ for a point $ P \in X $ does not depend on
    $P$; for a general point $P$ it is just the sum of the weights of the
    direction vectors for all edges of $ \Gamma $ that map some point to $P$
    under $h$. In particular, taking for $P$ a point far out on an unbounded
    edge of $X$, we see that this degree depends only on $ \Sigma $ and not on
    $h$. It will be called the \emph{degree} $ \deg h $ of $h$ and is the
    tropical counterpart of the notion of degree of a morphism between smooth
    curves in algebraic geometry. The degree of $ h^* K_X $ is now just $ \deg
    h \cdot \deg K_X $; in particular by the above it depends only on $ \Sigma
    $ and not on $h$.
  \item \label{rem-degree-kx:b}
    If $X$ is a hypersurface in $ \RR^N $ a local computation shows that $
    K_X = X \cdot X $, and hence by the projection formula the degree of $ h^*
    K_X $ on $ \Gamma $ is the same as the degree of $ h_* \Gamma \cdot X $ on
    $ \RR^N $. But the $1$-cycle $ h_* \Gamma $ in $ \RR^N $ is rationally
    equivalent to its so-called \emph{recession fan}, \ie the fan obtained by
    shrinking all bounded edges to zero length. As this fan is determined by $
    \Sigma $ it follows also in this case that the degree of $ h^* K_X $
    depends only on $ \Sigma $ and not on $h$.
  \end{enumerate}
  We will therefore denote the degree of $ h^* K_X $ also by $ K_X \cdot \Sigma
  $.
\end{remark}

\section{Gluing moduli spaces} \label{sec:gluing}

Throughout this section, let $X\subset \RR^N$ be a smooth rational curve or a
smooth hypersurface as in Example \ref{ex-smooth}. Aiming at enumerative
applications, we want to construct tropical analogues of the algebraic moduli
spaces of stable maps to a variety, or in other words generalizations of the
tropical moduli spaces $ \mm {\RR^N}{\Sigma} $ of Construction
\ref{constr-curve} to other target spaces than $ \RR^N $. The naive approach
would simply be to use the subset
  \[ \mx X\Sigma :=\{(\Gamma,x_1,\dots,x_n,h) \in \mm {\RR^N}{\Sigma} :
     h(\Gamma)\subset X\} \]
of tropical curves mapping to $X$. As $ K_X \cdot \Sigma $ is independent of
the curves in this space by Remark \ref{rem-degree-kx}, the algebro-geometric
analogue tells us that we expect $ \mx X\Sigma $ to be of dimension $ \dim X +
|\Sigma|-3-K_X \cdot \Sigma $ (in fact, this independence is the reason why we
restrict to the curve and hypersurface cases in this paper). However, just as
in the algebraic case, the actual dimension of this space might be bigger, as
the following example shows.

\begin{example} \label{ex-predim}
  Let $ X=L^2_1 $ be the tropical line in $\R^2$ as in Example \ref{ex-linear},
  and consider degree-$2$ covers of $X$, \ie $ \Sigma=(e_0,e_0,e_1,e_1,e_2,e_2)
  $. The following pictures show combinatorial types of tropical curves in $
  \mx X\Sigma $. Edges are labeled with their weight if it is not $1$, and
  their directions in the picture indicate which cell of $X$ they are mapped
  to. The edge drawn with a dashed line is contracted to a point.

  \begin{center} \begin{tikzpicture}
    \draw (0,2)--(1.6,2);
    \draw (1.6,2)--(2.5,3);

    \draw (1.6,2)--(1.6,1.7);
    \draw (1.6,0.8)--(1.6,1.5);

    \draw (0,1.5)--(1.5,1.5);
    \draw (1.5,1.5)--(2.5,2.5);
    \draw (1.5,1.5)--(1.5,0.5);

    \draw[dashed] (0.7,2)--(0.7,1.5);
    \node at (1,.5) {$\alpha_1$};

    \begin{scope}
      \draw (5,2.05)--(6.5,2);
      \draw (5,1.95)--(6.5,2);
      \draw (6.5,2)--(7.5,3);
      \draw (7,2.5)--(7.55,2.9);
      \draw (6.5,2)--(6.5,1);
      \draw (6.5,1)--(6.55,0.5);
      \draw (6.5,1)--(6.45,0.5);
      \node at (6,.5) {$\alpha_2$};
      \node at (6.35,1.4) {\tiny 2};
      \node at (6.7,2.4) {\tiny 2};
    \end{scope}
  \end{tikzpicture}

  \begin{tikzpicture}
    \begin{scope}[xshift=-10cm]
      \draw (10,2.05)--(10.75,2);
      \draw (10,1.95)--(10.75,2);
      \draw (10.75,2)--(11.5,2);

      \draw (11.5,2)--(12.5,3);
      \draw (12,2.5)--(12.55,2.9);

      \draw (11.5,2)--(11.5,1);
      \draw (11.5,1)--(11.55,0.5);
      \draw (11.5,1)--(11.45,0.5);
      \node at (11.35,1.4) {\tiny 2};
      \node at (11.7,2.4) {\tiny 2};
      \node at (11.1,2.15) {\tiny 2};
      \node at (11,.5) {$\alpha_3$};
    \end{scope}

    \begin{scope}[xshift=5cm]
      \draw (0,2.1)--(0.5,2);
      \draw (0,1.9)--(0.5,2);
      \draw (1,2)--(0.5,2);
      \node at (0.75,2.15) {\tiny 2};

      \draw (1,2)--(1.5,2.1);
      \draw (1,2)--(1.4,1.96);
      \draw (1.7,1.93)--(1.6,1.94);		
				
      \draw (1.5,2.1)--(1.5,.6);
      \draw (1.5,2.1)--(2.5,3.1);
		
      \draw (1.7,1.93)--(1.7,.5);
      \draw (1.7,1.93)--(2.7,2.93);	
      \node at (1,.5) {$\alpha_4$};		
    \end{scope}
  \end{tikzpicture} \end{center}

  Note that in $ \alpha_1 $ and $ \alpha_4 $ the lengths of the bounded edges
  are not independent, since in both cases the two horizontal bounded edges
  adjacent to the origin must have the same length. Hence the types $ \alpha_1
  $, $ \alpha_2 $, and $ \alpha_4 $ are described by $2$-dimensional cells in
  the moduli spaces, whereas $ \alpha_3 $ is $3$-dimensional (with $ \alpha_2
  $ as one of its faces). As we expect our tropical moduli space to have
  dimension $2$ (equal to the space of algebraic degree-$2$ covers of $ \PP^1
  $), we see that $ \mx X\Sigma $ has too big dimension.
\end{example}

Our first aim is therefore to define a suitable subset $ \mm X\Sigma $ of $ \mx
X\Sigma $ of the expected dimension. We fix the polyhedral complex structure on
$X$ to be the unique coarsest one (which exists since $X$ is smooth), and
choose the polyhedral complex structure on $ \mx X\Sigma $ as follows.

\begin{notation}[Curves in $X$] \label{not-combtype}
  Let $ (\Gamma,x_1,\dots,x_n,h) \in \mx X\Sigma $. In the following, all
  isolated points of $ h^{-1}(\sigma) $ for a cell $ \sigma \in X $ will be
  considered as (possibly additional $2$-valent) vertices of $ \Gamma $. The
  interior of every edge of $ \Gamma $ then maps to the relative interior of a
  unique cell of $X$, and we include this information in the combinatorial type
  of a curve in $X$.

  For such a combinatorial type $ \alpha $, we denote the set of all curves in
  $ \mx X\Sigma $ of this type by $ \calM(\alpha) $. These are partially open
  polyhedra, and their closures $ \overline{\calM(\alpha)} $ give $ \mx
  X\Sigma $ the structure of a polyhedral complex. We write $\beta\geq\alpha$
  for two combinatorial types with $\overline{\calM(\beta)}\supset
  \calM(\alpha)$, and say in this case that $ \alpha $ is a \emph{face} of $
  \beta $. If in addition $ \beta\neq\alpha $ we call $\beta$ a
  \emph{resolution} of $\alpha$.
\end{notation}

\begin{notation}[Local curves] \label{not-local}
  Let $ C = (\Gamma,x_1,\dots,x_n,h) \in \mx X\Sigma $ be a tropical curve in
  $X$, and fix a vertex $V$ of $ \Gamma $. We can restrict $C$ to the local
  situation around $V$ and obtain the following data, all written with an index
  $V$:
  \begin{enumerate}
  \item \label{not-local:a}
    The collection of all vectors $ v(e,V) $ for $ e \ni V $ as in Construction
    \ref{constr-curve} is called the \emph{local degree} $ \Sigma_V $ of the
    curve at $V$.
  \item \label{not-local:b}
    The \emph{star} of $X$ at $ h(V) $ will be denoted $ X_V $; it is a shifted
    tropical fan.
  \item \label{not-local:c}
    By our assumption on $X$, we know that $ X_V $ is isomorphic to $
    L^{k_V}_{r_V} \times \R^{m_V} $ for unique $ k_V,r_V,m_V $ (where $ 0 < r_V
    < k_V $ unless $ (k_V,r_V)=(0,0) $). Note that the degree of the canonical
    divisor as in Remark \ref{rem-degree-kx} then splits up as
      \[ K_X \cdot \Sigma = \sum_V K_{X_V} \cdot \Sigma_V, \]
    with the sum taken over all vertices of $ \Gamma $.
  \item \label{not-local:d}
    Let $ C_V \in \mx {X_V}{\Sigma_V} $ be the curve in $ X_V $ with one vertex
    mapping to $ h(V) $, and unbounded ends of directions $ \Sigma_V $. We
    refer to $ C_V $ as a \emph{local curve}. Its combinatorial type $ \alpha_V
    $ will be called the \emph{trivial combinatorial type} of $V$ in $ \mx
    {X_V}{\Sigma_V} $. We will refer to a resolution of $ \alpha_V $ also as a
    \emph{resolution of $V$}.
  \end{enumerate}
\end{notation}

\begin{definition} \label{def-dim}
  Let $ V \in \Gamma $ be a vertex of a curve $ (\Gamma,x_1,\dots,x_n,h) \in
  \mx X\Sigma $. Using Notation \ref{not-local}, we define
  \begin{enumerate}
  \item \label{def-dim:a}
    the \emph{virtual dimension} of $V$ as
    $ \vdim(V)=\val(V)- K_{X_V} \cdot \Sigma_V +\dim X-3 $,
  \item \label{def-dim:b}
    the \emph{resolution dimension} of $V$ as
    $ \rdim(V)=\val(V)- K_{X_V} \cdot \Sigma_V +r_V-3 $,
  \item \label{def-dim:c}
    the \emph{classification number} of $V$ as
    $ c_V=\val(V)+r_V $.
  \end{enumerate}
\end{definition}

\begin{example} \label{ex-dim}
  Let $ X=L^4_3 $, and consider the combinatorial types $ \alpha > \beta $ of
  curves in $X$ as shown in the picture below. In the type $ \beta $, the
  vertex $V$ is mapped to the origin, and the unbounded ends have
  directions $ e_0+e_1+e_2 $ (with weight $6$), and twice $ e_3 $ and $ e_4 $
  (with weight $3$). The type $ \alpha $ has the same directions of the ends,
  $ V_1 $ and $ V_2 $ map to the origin, and consequently $ V_0 $ to a positive
  multiple of $ e_0+e_1+e_2 $.

  \begin {center} \input {pics/rdim} \end {center}

  We then have $ \vdim(V)=\rdim(V)=-1 $, $ \vdim(V_1)=\rdim(V_1)=\vdim(V_2)=
  \rdim(V_2)=0 $, $ \vdim(V_0)=3 $, and $ \rdim(V_0)=0 $. Moreover, $ c_V = 8
  $, $ c_{V_1} = c_{V_2} = 6 $, and $ c_{V_0} = 3 $.

  The virtual dimension of a vertex $V$ can be thought of as the expected
  dimension of the moduli space of curves in $ X_V \cong L^{k_V}_{r_V} \times
  \R^{m_V} $ of degree $ \Sigma_V $. It includes the dimension of a lineality
  space coming from translations in $ \RR^{m_V} $, which is subtracted from the
  virtual dimension to obtain the resolution dimension (as can be seen in the
  example above for $ V_0 $, which can locally be moved in its $3$-dimensional
  cell spanned by $ e_0,e_1,e_2 $). The classification number of Definition
  \ref{def-dim} \refx{def-dim:c} is a useful number for inductive proofs
  because it is always non-negative and becomes smaller in resolutions, as the
  following lemma shows.
\end{example}

\begin{lemma} \label{lem-classnumber}
  Let $\alpha$ be a resolution of a vertex $V$ of a curve in $X$. Then the
  classification number of every vertex $W$ of $ \alpha $ is smaller than $ c_V
  $.
\end{lemma}

\begin{proof}
  Note first that $ c_W \le c_V $ since both summands in the definition of the
  classification number cannot get bigger when passing from $V$ to $W$:
  \begin{enumerate}
  \item \label{lem-classnumber:a}
    As $ \alpha $ is a tree, shrinking an edge to zero length will merge
    two vertices into one, whose valence is the sum of the original valences
    minus $2$. In particular, since the trivial combinatorial type of $V$ is
    obtained from $ \alpha $ by a sequence of such processes, we conclude that
    $ \val(V) \ge \val(W) $, with equality if and only if all vertices of $
    \alpha $ except $W$ have valence $2$.
  \item \label{lem-classnumber:b}
    If $W$ is mapped to the relative interior of a cone $ \sigma_W $, we
    have $ r_W = \dim X - \dim \sigma_W $. The analogous statement holds for
    $V$, and hence
      \[ r_V = \dim X - \dim \sigma_V \ge \dim X - \dim \sigma_W = r_W \]
    since $ \sigma_V $ is a face of $ \sigma_W $.
  \end{enumerate}
  If we had equality for both numbers, all vertices $W'$ of $ \alpha $
  except $W$ must have valence $2$ by \refx{lem-classnumber:a}. Moreover, $W$
  and $V$ have to lie in the same cell by \refx{lem-classnumber:b}. Hence $
  \alpha $ is the same combinatorial type as the trivial type $V$ after
  removing all $2$-valent vertices $W'$. But this means that all two-valent
  vertices lie in the interior of an edge that is completely mapped to one cell
  of $X$. As this is excluded by definition, there can actually be no
  $2$-valent vertices. Hence $ \alpha $ is the trivial type $V$, in
  contradiction to $ \alpha $ being a resolution of $V$.
\end{proof}

The idea to construct the desired moduli spaces of curves in $X$ is now as
follows.
\begin{enumerate}
\item \label{idea:a}
  Vertices of negative resolution dimension should not be admitted, since
  they correspond (locally, and modulo their lineality space) to an algebraic
  moduli space of curves of negative virtual dimension. We will exclude them in
  Definition \ref{def-mx}. For the case of curves, this corresponds to the
  Riemann-Hurwitz condition as \eg in \cite{BBM10,Cap12,CMR14,BM13}.
\item \label{idea:b}
  Vertices of resolution dimension $0$ correspond (again locally and modulo
  their lineality space) to a $0$-dimensional algebraic moduli space.
  Hence the curves in the corresponding tropical moduli spaces should not allow
  any resolutions, \ie these moduli spaces will consist of only one cell, whose
  weight is the degree of the algebraic moduli space.

  In this paper, we will only consider the tropical situation. The weights of
  the cells for vertices of resolution dimension $0$ are therefore considered
  as initial input data for our constructions, as in Definition
  \ref{def:mod-data} below.
\item \label{idea:c}
  Vertices of positive resolution dimension will lead to curves that allow
  resolutions, and whose moduli spaces therefore consider of several cells.
  These spaces can be obtained recursively by gluing from the initial ones of
  \refx{idea:b} using Construction \ref{con:gluing}, so that no additional
  input data is required for these cases.

  However, as for this paper the initial weights of \refx{idea:b} are arbitrary
  numbers a priori, they need to satisfy some compatibility conditions in order
  for the gluing process to lead to a well-defined space. These conditions are
  encoded in the notion of a good vertex in Definition \ref{def:good}. They are
  originally formulated in every resolution dimension; we will see in Corollary
  \ref{cor:good1dim} however that only the conditions in resolution dimension
  $1$ are relevant since the others follow from them.

  In this recursive construction of the moduli spaces, the following example
  shows that the resolution dimension is not necessarily strictly increasing.
  We will therefore use the classification number for these purposes.
\end{enumerate}

\begin{example} \label{ex-recursive}
  Consider again two combinatorial types $ \alpha > \beta $ as in Example
  \ref{ex-dim}, however in $ X=L^3_2 $ and with directions as indicated in the
  following picture.

  \begin {center} \input {pics/rdim2} \end {center}

  Then $ \rdim (V) = \rdim (V_0) = \rdim (V_1) = 1 $ and $ \rdim (V_2) = 0 $,
  \ie all resolution dimensions are non-negative and hence will be admitted.
  The weight for the type $ \alpha $ will be defined by a gluing procedure over
  its three vertices in Construction \ref{con:gluing}. However, $ V_1 $ has the
  same resolution dimension as $V$, so that a recursive construction over the
  resolution dimension would not work.
\end{example}

\begin{definition}[The moduli space $ \mm X\Sigma $ as a polyhedral complex]
    \label{def-mx} ~
  \begin{enumerate}
  \item \label{def-mx:a}
    An \emph{admissible} combinatorial type of a curve in $X$ is a
    combinatorial type $\alpha$ such that for all vertices $V$ in $\alpha$ we
    have $\rdim(V)\geq 0$.
  \item \label{def-mx:b}
    We denote by $ \mm X\Sigma $ the polyhedral subcomplex of $ \mx X\Sigma $
    consisting of all closed cells $ \overline {\calM(\alpha)} $ such that
    $ \alpha $ and all its faces are admissible, and
      \[ \qquad\quad
         \dim\calM(\alpha) = \vdim \mm X\Sigma
         := |\Sigma| - K_X\cdot\Sigma +\dim X - 3. \]
    Note that this dimension is just $ \vdim(V) $ if $ \alpha $ is a
    combinatorial type with just one vertex $V$.
  \item \label{def-mx:c}
    The \emph{neighborhood} of a combinatorial type $ \alpha $ in $ \mm X\Sigma
    $ is defined as
      \[ \qquad\quad
         \calN(\alpha) := \bigcup_{\beta\geq\alpha} \calM(\beta), \]
    where the union is taken over all combinatorial types $ \beta \ge \alpha $
    in $ \mm X\Sigma $.
  \end{enumerate}
  In the following, we will apply this definition also to the case of local
  curves as in Notation \ref{not-local}, in order to obtain moduli spaces
  $ \mm{X_V}{\Sigma_V} $.
\end{definition}

Example \ref{ex-dim} shows that faces of admissible types need not be
admissible again, so that the condition of all faces of $ \alpha $ being
admissible in Definition \ref{def-mx} \refx{def-mx:b} is not vacuous.

\begin{definition}[Moduli data] \label{def:mod-data}
  \emph{Moduli data} for a smooth variety $X$ are a collection $(\omega_V)_V$
  of weights in $\Q$ for every vertex $V$ of a curve in $X$ with $\rdim(V)=
  0$. All subsequent constructions and results in this section will depend on
  the choice of such moduli data.
\end{definition}

\begin{construction}[Vertices of resolution dimension $0$] \label{con:rdim0}
  Let $V$ be a vertex of a local curve in $ \mm {X_V}{\Sigma_V} $. As $ X_V
  \cong L^{k_V}_{r_V}\times\R^{m_V} $, the trivial combinatorial type of $V$ in
  $ \mm{X_V}{\Sigma_V} $ has dimension $ m_V $, corresponding to a translation
  along $ \RR^{m_V} $. But $ \vdim(V) = \rdim(V) + m_V $, and thus the trivial
  combinatorial type is maximal in $ \mm{X_V}{\Sigma_V} $ if and only if $
  \rdim(V) = 0 $. In this case, the moduli space $ \mm{X_V}{\Sigma_V} $
  consists of only one cell, and we equip it with the weight $ \omega_V $ from
  our moduli data.
\end{construction}

To define the weights on $ \mm X\Sigma $ in general we need Definition
\ref{def:good} and Construction \ref{con:gluing}, which depend on each other
and work in a combined recursion on the classification number of vertices. The
following definition of a good vertex of a certain classification number thus
assumes that good vertices of smaller classification number have already been
defined. Moreover, for every combinatorial type $ \alpha $ in a local moduli
space $ \mm{X_V}{\Sigma_V} $ all of whose vertices have smaller classification
number and are good it assumes from Construction \ref{con:gluing} below that
there is a \emph{gluing cycle} $ \calZ(\alpha) $ on the neighborhood $
\calN(\alpha) $.

\begin{definition}[Good vertices] \label{def:good}
  Let $V$ be a vertex of a (local) curve in $ \mm {X_V}{\Sigma_V} $, of
  classification number $ c_V $. By recursion on $ c_V $, we define $V$ to be
  a \emph{good} vertex if the following three conditions hold.
  \begin{enumerate}
  \item \label{def:good:a}
    Every vertex of every resolution $ \alpha $ of $V$ in $ \mm{X_V}{\Sigma_V}
    $ (which has classification number smaller than $c_V$ by Lemma
    \ref{lem-classnumber}) is good. (There is then a gluing cycle $
    \calZ(\alpha) $ on $ \calN(\alpha) $ by Construction \ref{con:gluing}).
  \item \label{def:good:b}
    $ \mm {X_V}{\Sigma_V} $ is a tropical cycle with the following weights:
    \begin{itemize}
    \item If $ \rdim(V)=0 $ we equip the unique cell of $ \mm {X_V}{\Sigma_V} $
      with the weight from the moduli data as in Construction \ref{con:rdim0}.
    \item If $ \rdim(V)>0 $ the maximal types $ \alpha $ in $ \mm
      {X_V}{\Sigma_V} $ are not the trivial one, \ie they are resolutions of
      $V$. The weights on the corresponding cells $ \calM(\alpha)=\calN(\alpha)
      $ are then the ones of the gluing cycles $ \calZ(\alpha) $ as in
      \ref{def:good:a}.
    \end{itemize}
  \item \label{def:good:c}
    For every resolution $ \alpha $ of $V$ in $ \mm{X_V}{\Sigma_V} $ and every
    maximal type $ \beta \ge \alpha $ in $ \mm{X_V}{\Sigma_V} $ (which is then
    also a resolution of $V$), the weight of the cell $ \calM(\beta) $ is the
    same in the gluing cycles $ \calZ(\alpha) $ and $ \calZ(\beta) $.
  \end{enumerate}
\end{definition}

\begin{construction}[The gluing cycle $ \calZ(\alpha) $] \label{con:gluing}
  Fix a (not necessarily maximal) combinatorial type $\alpha$ in a moduli space
  $ \mm X\Sigma $ as in the picture below on the left, and assume that all its
  vertices are good. We will now construct a cycle $ \calZ(\alpha) $ of
  dimension $ \vdim \mm X\Sigma $ on the neighborhood $ \calN(\alpha) $. An
  important example of this is when $ \alpha $ is a resolution of a vertex in a
  local moduli space. More technical details on this construction can be found
  in \cite[Construction 1.5.13]{Och13}.

  For each vertex $V$ of $ \alpha $ let $ \sigma_V^\circ $ be the open cell of
  $X$ in which $V$ lies, and let $ X(V) = \bigcup_{\sigma \supset \sigma_V}
  \sigma^\circ \subset X $; it is an open neighborhood of $ \sigma_V $.
  Similarly, for each edge $e$ of $ \alpha $ let $ \sigma_e^\circ $ be the open
  cell of $X$ in which $e$ lies, and set $ X(e) = \bigcup_{\sigma \supset
  \sigma_e} \sigma^\circ $.

  \begin {center} \input {pics/glue} \end {center}

  First we will construct local moduli spaces $ \calM_V $ for each vertex $V$
  of $ \alpha $. As $V$ is good the local moduli space $ \mm{X_V}{\Sigma_V} $
  is a tropical cycle by Definition \ref{def:good} \refx{def:good:b}. Let $
  \mmp{X_V}{\Sigma_V} $ be the corresponding moduli cycle of curves with
  bounded ends, \ie the pull-back of $ \mm{X_V}{\Sigma_V} $ under the quotient
  map that forgets the lengths of the ends as in Construction
  \ref{constr-m0nird}. We denote by
    \[ \calM_V = \{ (\Gamma,x_1,\dots,x_n,h) \in \mmp{X_V}{\Sigma_V} :
       h(\Gamma) \subset X(V) \} \]
  its partially open polyhedral subcomplex of all curves with bounded ends
  that lie entirely in $ X(V) $. A typical element of $ \calM_V $ is a curve
  piece as shown in the picture above on the right. (Note that these pieces
  might also be resolutions of the corresponding vertices.)

  Now we glue these pieces $ \calM_V $ together. For each bounded edge $e$
  of $ \alpha $, joining two vertices $ V_1 $ and $ V_2 $ (as in the picture),
  there are corresponding bounded ends in $ \calM_{V_1} $ and $ \calM_{V_2} $.
  Let
    \[ \ev_e : \prod_V \calM_V \to X(e) \times X(e) \]
  be the evaluation map at these two ends, where the product is taken over all
  vertices $V$ of $ \alpha $. The product
    \[ \prod_e \ev_e^* \Delta_{X(e)} \]
  over all pull-backs of the diagonals $ \Delta_{X(e)} $ along these evaluation
  maps as in Appendix \ref{sec:diagonal} can also be written as the product $
  \prod_e \ev_e^* \Delta_X $ along the extended evaluation maps to $ X \times X
  $ by Remark \ref{rem:pull-smooth}. We will therefore abbreviate it by $ \ev^*
  \Delta_X $; it is a cocycle on $ \prod_V \calM_V $. The support of the cycle
  $ \ev^* \Delta_X \cdot \prod_V \calM_V $ then consists of points
  corresponding to curve pieces that can be glued together to a curve in $X$,
  \ie so that \eg the positions of the dots in the picture above coincide.

  However, these curve pieces still carry the information about the position of
  the gluing points. In order to forget these positions, we take a quotient as
  follows. For each bounded edge $e$ of $ \alpha $ between two vertices $ V_1 $
  and $ V_2 $ as above, there are two vectors $ u_{V_1} $ and $ u_{V_2} $ in
  the lineality spaces of $ \calM_{V_1} $ and $ \calM_{V_2} $, respectively,
  that parametrize the lengths of these ends as in Constructions
  \ref{constr-m0n} and \ref{constr-m0nb}. The vector $ u_e := u_{V_1}-u_{V_2} $
  is then in the lineality space of $ \prod_V \calM_V $, and we denote by $
  L_\alpha $ the vector space spanned by the vectors $ u_e $ for all bounded
  edges $e$ of $ \alpha $. Let $q$ be the quotient map by $ L_\alpha $; the
  quotient $ q \big( \ev^* \Delta_X \cdot \prod_V \calM_V \big) $ then does not
  contain the information about the gluing points any more.

  From this cycle we now obtain an injective morphism that considers the curve
  pieces as a glued curve in $X$. It can be defined as
  \begin{align*}
    f: q \big( \ev^* \Delta_X \cdot \prod_V \calM_V \big)
         &\to \calM'_{0,n}(\RR^N,\Sigma) \\
       \big( (d_{i,j})_{\{i,j\} \in R_V},a^V \big)_V &\mapsto 
      \Bigg( \bigg( \sum_{\{k,l\} \in R_{i,j}} d_{k,l} \bigg)_{\{i,j\} \in R},
      a^{V_0} \Bigg),
  \end{align*}
  where
  \begin{itemize}
  \item $ a^V $ denotes the coordinates of the root vertex in the local moduli
    space $ \calM_V $ for $V$, of which we choose $ V_0 $ as the root vertex in
    $ \calM'_{0,n}(\RR^N,\Sigma) $;
  \item $ d_{i,j} $ denotes the distance coordinates on the moduli spaces, with
    $ R_V $ and $R$ the index sets of all pairs of ends in the local moduli
    spaces $ \calM_V $ and the full moduli space $ \calM'_{0,n}(\RR^N,\Sigma)
    $, respectively;
  \item $ R_{i,j} \subset \bigcup_V R_V $ for $ \{ i,j \} \in R $ is the set of
    all pairs of ends of the local curves in the moduli spaces $ \calM_V $
    that lie on the unique path between the ends $ x_i $ and $ x_j $.
  \end{itemize}
  The push-forward cycle
    \[ \calZ'(\alpha)
       := f_* q\bigg(\ev^* \Delta_X \cdot\prod_V\calM_V\bigg)
       := f_* \Bigg( q\bigg(\ev^* \Delta_X \cdot\prod_V\calM_V\bigg) \Bigg)
       \]
  in $ \calM'_{0,n}(\R^N,\Sigma) $ will be called the \emph{gluing cycle}
  (with bounded ends) for $ \alpha $. Finally, taking the quotient by the
  lineality space corresponding to the ends of $ \alpha $, we obtain a gluing
  cycle $ \calZ(\alpha) $ (with unbounded ends) in $ \calM_{0,n}(\R^N,\Sigma)
  $. The cells of maximal dimension come with a natural weight in this
  construction, which we will call the \emph{gluing weight}. It is not clear a
  priori that this weight is independent of the choice of $\alpha$, but it will
  turn out to be so in Theorem \ref{thm-balancing}.
\end{construction}

\begin{lemma} \label{lem:dimtrop}
  Assume that all vertices occurring in a combinatorial type $ \alpha
  $ in $ \mm X\Sigma $ are good. Then
    \[ \dim \calZ(\alpha)=\dim X + |\Sigma|-3-K_X\cdot\Sigma. \]
\end{lemma}

\begin{proof}
  By definition we have
    $$ \calZ'(\alpha)=f_*q\bigg(\ev^*\Delta_{X} \cdot \prod_V\calM_V\bigg). $$
  Let $s$ denote the number of vertices $V$ in $ \alpha $, so that $s-1$ is
  the number of bounded edges. As the push-forward preserves dimensions, we
  only have to compute the dimension of $ q \big( \ev^* \Delta_{X} \cdot
  \prod_V\calM_V \big) $. We have that
  \begin{align*}
    \dim \prod_V\calM_V
    & = \sum_V \left(2\val(V)- K_{X_V} \cdot \Sigma_V +\dim X-3\right) \\
    & = s\dim X-K_X\cdot\Sigma+\sum_V(\val(V)-3)+\sum_V\val(V) \\
    & = s\dim X-K_X\cdot\Sigma+2 |\Sigma|-4+s,
  \end{align*}
  where we used that for a tree the number of vertices equals
  $|\Sigma|-2-\sum_V(\val(V)-3)$. The cycle $ \ev^* (\Delta_{X}) \cdot
  \prod_V \calM_V $ has codimension $(s-1)\dim X$, and taking the quotient
  $q$ eliminates another $s-1$ dimensions. Passing from $ \calZ'(\alpha) $ to $
  \calZ(\alpha) $ reduces the dimension by $|\Sigma|$, so in total we get
  \begin{align*}
    \dim \calZ(\alpha)
    &= s\dim X-K_X\cdot\Sigma+2 |\Sigma|-4+s -(s-1) \dim X - (s-1) -|\Sigma| \\
    &= \dim X + |\Sigma|-3-K_X\cdot\Sigma.
  \end{align*}
\end{proof}

\begin{lemma} \label{lem:lin-space-gluing}
  Assume that all vertices occurring in a combinatorial type $ \alpha
  $ in $ \mm X\Sigma $ are good. If the gluing cycle $ \calZ(\alpha) $ is not
  zero, the support of $ \calZ(\alpha) $ contains $ \calM(\alpha) $ and is
  contained in $ \calN(\alpha) $ (so in particular also in $ \calM_0(X,\Sigma)
  $).
\end{lemma}

\begin{proof}
  In the notation of Construction \ref{con:gluing}, consider the set $
  q^{-1}(f^{-1}(\calZ'(\alpha)) $ of all curve pieces that glue to the
  combinatorial type $ \alpha $. Under each evaluation map $ \ev_e $, this set
  maps to the lineality space of $ \Delta_{X(e)} $ in $ X(e) \times X(e) $. The
  functions used for cutting out the diagonal in Construction
  \ref{con:diagonal} all contain this lineality space in their own lineality
  space, so their pull-backs are linear on the above set and hence do not
  subdivide it. These properties are preserved under push-forward with $f$ and
  taking the quotient $q$, and thus $ \calZ(\alpha) $ contains all of $
  \calM(\alpha) $ if it is non-zero.

  By construction, the other points in the gluing cycle $ \calZ(\alpha) $
  correspond to curves obtained by resolving each vertex of $ \alpha $, and
  thus to resolutions of $ \alpha $. Hence, $ \calZ(\alpha) $ is contained in
  the union of all neighboring cells of $ \calM(\alpha) $, and thus by Lemma
  \ref{lem:dimtrop} in $ \calN(\alpha) $.
\end{proof}

\begin{definition}[The moduli space $ \mm X\Sigma $ as a weighted polyhedral
    complex] \label{def:m0nX}
  Assume that all vertices occurring in curves in $ \mm X\Sigma $ are good.
  By Lemma \ref{lem:dimtrop} and \ref{lem:lin-space-gluing}, each maximal cell
  $ \alpha $ of $ \mm X\Sigma $ is also a maximal cell in the gluing cycle $
  \calZ(\alpha) $. We define the weight of this cell of $ \mm X \Sigma $ to be
  the corresponding gluing weight of Construction \ref{con:gluing}.
\end{definition}

The aim of this section is to show that these weights make $ \mm X\Sigma $
into a tropical variety, \ie balanced, if all vertices of resolution dimension
$1$ are good. Examples can be found in Section \ref{sec-lines} and \cite{GMO}.

\begin{theorem} \label{thm-balancing}
  Assume that all vertices in a combinatorial type $ \alpha $ in $ \mm X\Sigma
  $ are good. If $ \beta $ is a resolution of $ \alpha $ in $ \mm X\Sigma $ (so
  that in particular $ \calN(\beta) \subset \calN(\alpha) $) then the weight of
  every maximal cell in the gluing cycle $ \calZ(\beta) $ in $ \calN(\beta) $
  agrees with the weight of the same cell in $ \calZ(\alpha) $ in $
  \calN(\alpha) $.

  In particular, if $ \alpha $ is of virtual codimension $1$ then the cycle $
  \mm X\Sigma $ with the weights of Definition \ref{def:m0nX} is balanced at $
  \alpha $.
\end{theorem}

\begin{proof}
  This is basically a straight-forward recuction proof; however, we have to pay
  attention to several intersection-theoretical details.
  We start by describing the gluing cycle $ \calZ'(\beta) $ as in Construction
  \ref{con:gluing}. For every vertex $V$ of $ \alpha $, let $ J_V $ be the set
  of vertices of $ \beta $ that degenerate to $V$ in $ \alpha $, so that $
  \cup_V J_V $ is the set of all vertices of $ \beta $. We denote by $ \EV^*
  \Delta_X $ the product over all pull-backs of the diagonals along evaluation
  maps belonging to the edges of $ \beta $, by $Q$ the quotient map on $ \EV^*
  \Delta_X \cdot \prod_V \prod_{W \in J_V} \calM_W $ forgetting the gluing
  points along the bounded edges, and by $F$ the morphism embedding the
  resulting cycle to $ \mmp X\Sigma $. Then by Construction \ref{con:gluing}
  the gluing cycle $ \calZ'(\beta) $ is given by
    \[ \calZ'(\beta)
       = F_*Q \bigg(\EV^* \Delta_{X} \cdot \prod_V \prod_{W \in J_V}\calM_W
       \bigg), \]
  where the first product is taken over all vertices $V$ of $ \alpha $. We will
  now decompose the maps $Q$, $F$, and $\EV$ into contributions coming from the
  vertices of $ \alpha $ as follows. For each vertex $V$ of $ \alpha $,
  let $ I_V $ be the set of bounded edges of $ \beta $ contracting to $V$ in $
  \alpha $. We denote by $ q_V $ the quotient map that forgets the gluing
  points on these edges, and by $ f_V $ the morphism that embeds the cycle
  $ \prod_{g \in I_V} \ev_g^* \Delta_X \cdot \prod_{W \in J_V} \calM_W $ in the
  local moduli space $ \mmp {X_V}{\Sigma_V} $. Furthermore, denote by $q$ and
  $f$ the quotient and embedding maps for the gluing cycle for $ \alpha $,
  respectively. With $ \tilde q = \prod_V q_V $ and $ \tilde f = \prod_V f_V $
  we can then write
    \[ \calZ'(\beta)
       = f_*q\bigg(\tilde{f}_*\tilde{q}\bigg( \prod_e \ev_e^* \Delta_X
       \cdot \prod_V \prod_{g \in I_V} \ev_g^* \Delta_X
       \cdot \prod_V \prod_{W \in J_V} \calM_W \bigg)\bigg), \]
  by Lemma \ref{lem:quot-push}, with the product over $e$ running over all
  edges of $ \alpha $. By the compatibility of push-forwards and quotient maps
  with diagonal pull-backs (see Lemma \ref{lem:diag-compat}), we may rewrite
  this as
  \begin{align*}
    \calZ'(\beta)
    &= f_*q\bigg( \prod_e \ev_e^* \Delta_X \cdot \tilde f_* \tilde q \bigg(
       \prod_V \prod_{f \in I_V} \ev_f^* \Delta_X
       \cdot \prod_V \prod_{W \in J_V} \calM_W \bigg)\bigg) \\
    &= f_*q\bigg( \ev^* \Delta_X \cdot \prod_V
       \underbrace{{f_V}_* q_V \bigg(
       \prod_{f \in I_V} \ev_f^* \Delta_X
       \cdot \prod_{W \in J_V} \calM_W \bigg)}_{=: \tilde \calM_V} \bigg),
       \tag{$*$}
  \end{align*}
  where $ \ev^* \Delta_X $ is the product over all pull-backs of the diagonals
  along evaluation maps belonging to edges of $ \alpha $. But now all vertices
  $V$ of $ \alpha $ are good by assumption, and therefore by Definition
  \ref{def:good} the weights of all maximal cells in the gluing cycle $ \tilde
  \calM_V $ agree with those of $ \calM_V $. Hence $ \tilde \calM_V $ is an
  open subcycle of $ \calM_V $, and as ($*$) is just the gluing construction
  for $ \alpha $ we conclude that every maximal cell of $ \calZ'(\beta) $ has
  the same weight in $ \calZ'(\alpha) $. Of course, this property is preserved
  when passing to unbounded ends in $ \calZ(\beta) $ and $ \calZ(\alpha) $.

  In particular, if $ \alpha $ is of virtual codimension $1$ and thus $ \beta $
  of virtual codimension $0$ in $ \mm X\Sigma $, we see that the cycle $ \mm
  X\Sigma $ is locally around $ \alpha $ given by the cycle $ \calZ(\alpha) $,
  and thus balanced.
\end{proof}

\begin{corollary} \label{cor-balancing}
  If all vertices that can appear in combinatorial types in $ \mm X\Sigma $ are
  good, then $ \mm X\Sigma $ is a tropical variety of dimension
    \[ \dim \mm X\Sigma = \dim X + |\Sigma|-3-K_X\cdot\Sigma \]
  with the weights of Definition \ref{def:m0nX}.
\end{corollary}

\begin{proof}
  Apply Theorem \ref{thm-balancing} to all combinatorial types of codimension
  $1$.
\end{proof}

So in order to obtain a tropical variety $ \mm X\Sigma $ from gluing we
will have to show that all vertices are good with respect to our given moduli
data. The following result tells us that we only have to do this in resolution
dimension $1$.

\begin{corollary} \label{cor:good1dim}
  If all vertices $V$ with $\rdim(V)=1$ in a given moduli space are good, then
  all vertices are good. 
\end{corollary}

\begin{proof}
  Let $V$ be a vertex in a given moduli space, with $ X_V \cong L^{k_V}_{r_V}
  \times \RR^{m_V} $. We will prove the statement of the lemma by induction on
  the classification number $ c_V $.
  
  If $ \rdim(V)=0 $ then $V$ does not admit any resolution in $ \mm X\Sigma $
  by Construction \ref{con:rdim0}, and hence $V$ is good. If $ \rdim(V)=1 $
  then $V$ is good by assumption. We can therefore assume that $ \rdim(V)>1 $.
  By Definition \ref{def-dim} this means that $ \vdim(V) $ is at least by $2$
  bigger than the dimension $ m_V $ of the lineality space in $ \mm
  {X_V}{\Sigma_V} $ coming from translations, and hence all combinatorial types
  $ \alpha $ of virtual codimension $1$ in $ \mm {X_V}{\Sigma_V} $ correspond
  to (non-trivial) resolutions of $V$.

  Condition \refx{def:good:a} of Definition \ref{def:good} of a good vertex now
  follows by induction on $ c_V $, condition \refx{def:good:b} by Theorem
  \ref{thm-balancing} applied to all these cells of virtual codimension $1$,
  and condition \refx{def:good:c} by Theorem \ref{thm-balancing} applied to all
  resolutions $ \alpha $ and maximal cells $ \beta > \alpha $.
\end{proof}

Taking Corollaries \ref{cor-balancing} and \ref{cor:good1dim} together, we thus
see that in order to obtain a well-defined moduli space $ \mm X\Sigma $ we only
have to check that all vertices of resolution dimension $1$ are good.

\section{Moduli spaces of lines in surfaces} \label{sec-lines}

In this section we want to construct the moduli spaces $ \mm X1 $ of lines in a
surface $X\subset\R^3$. By Corollary \ref{cor-balancing}, the dimension of
these spaces is the same as in the classical case, namely $3-\deg X$. So it is
empty for $\deg X>3$, and we obtain a finite number of lines counted with
multiplicities for $\deg X=3$.

Let us consider all possible local situations in such a surface. We want to use
decorations on the graph of the line to describe the vertices, as introduced in
\cite{Vig07}: A bold dot indicates that the line passes through a vertex of $X$
and a bold line indicates that the line passes through an edge of $X$. This
leads to the following combinatorial possibilities.

\begin{center} \begin{tikzpicture}
  \draw (0.2,1)--(1,1);
  \fill (0.6,1) circle (2pt);

  \draw (1.5,1)--(2.1,1);
  \draw (2.5,1.4)--(2.1,1);
  \draw (2.5,0.6)--(2.1,1);
  \fill (2.1,1) circle (2pt);

  \draw (3,1)--(3.6,1);
  \draw (4,1.4)--(3.6,1);
  \draw (4,0.6)--(3.6,1);
  \draw[line width=2pt] (3.6,1.4)--(3.6,0.6);

  \draw (4.5,0.5)--(5.5,1.5);
  \draw (4.5,1.5)--(5.5,0.5);
  \fill (5,1) circle (2pt);

  \draw (6,0.5)--(7,1.5);
  \draw (6,1.5)--(7,0.5);
  \draw[line width=2pt] (6.5,1.4)--(6.5,0.6);

  \draw (7.5,0.5)--(8.5,1.5);
  \draw (8,1)--(8.5,0.5);
  \draw[dashed] (7.5,1.5)--(8,1);
  \draw[line width=2pt] (8,1.4)--(8,0.6);
\end{tikzpicture} \end{center}

Here the difference between the last two decorations is that either one edge of
the line lies on the edge of $X$ (which is indicated by the dashed line in the
last picture), or all edges of the line point into maximal cells of $X$ (which
is indicated by the second picture from the right). Note that the pictures
above do not specify the combinatorial type completely, as there are in general
several possibilities for the directions of the ends.

\begin{remark}[Local degrees] \label{rem:locdeg}
  If $ d = \deg X $ then $ K_X \cdot \Sigma = d $ for the degree $ \Sigma $ of
  a line. This means that every local degree $ K_{X_V} \cdot \Sigma_V $ at a
  vertex $V$ can be at most $ d \le 3 $. Moreover, as $V$ has to be admissible,
  \ie must satisfy $ \rdim(V) \ge 0 $, Definition \ref{def-dim}
  \refx{def-dim:b} implies that $ \val(V) \ge K_{X_V} \cdot \Sigma_V +1 $ for
  the bold dot decorations and $ \val(V) \ge K_{X_V} \cdot \Sigma_V +2 $ for
  the bold line decorations. This leaves us with the following table, which
  lists the resolution dimensions of the possible types, and a name of the type
  in brackets. Impossible types are marked with ``X''. (The case above type (H)
  is excluded since there would have to be a maximal cell of $X$ containing two
  of the four edges, and hence $ K_{X_V} \cdot \Sigma_V $ cannot be $1$.)

  \begin{center} \begin{tabular}{c|c|c|c|c|c|c}
    $ K_{X_V} \cdot \Sigma_V $ &
    \begin{tikzpicture}
      \draw (0.2,1)--(1,1); \fill (0.6,1) circle (2pt);
      \draw[color=white] (0.2,0.6)--(1,0.6);
    \end{tikzpicture} &
    \begin{tikzpicture} \draw (1.5,1)--(2.1,1); \draw (2.5,1.4)--(2.1,1);
    \draw (2.5,0.6)--(2.1,1); \fill (2.1,1) circle (2pt); \end{tikzpicture} &
    \begin{tikzpicture} \draw (3,1)--(3.6,1); \draw (4,1.4)--(3.6,1);
    \draw (4,0.6)--(3.6,1); \draw[line width=2pt] (3.6,1.4)--(3.6,0.6);
    \end{tikzpicture} &
    \begin{tikzpicture} \draw (4.5,0.5)--(5.5,1.5);
    \draw (4.5,1.5)--(5.5,0.5);
    \fill (5,1) circle (2pt);
    \end{tikzpicture} &
    \begin{tikzpicture} \draw (6,0.5)--(7,1.5); \draw (6,1.5)--(7,0.5);
    \draw[line width=2pt] (6.5,1.4)--(6.5,0.6); \end{tikzpicture} &
    \begin{tikzpicture} \draw (7.5,0.5)--(8.5,1.5);
    \draw (8,1)--(8.5,0.5); \draw[dashed] (7.5,1.5)--(8,1);
    \draw[line width=2pt] (8,1.4)--(8,0.6); \end{tikzpicture}
    \\ \hline \hline
    1 & 0 (A)& 1 (B)& 0 (D)& 2 (E)& X &1 (I) \\ \hline
    2 & X& 0 (C)& X& 1 (F)&0 (H)&0 (J) \\ \hline
    3 & X& X& X& 0 (G)& X&X
  \end{tabular} \end{center}
\end{remark}

\begin{construction}[Moduli data for resolution dimension $0$]
    \label{con:surface0}
  For the vertices $V$ of resolution dimension $0$ in Remark \ref{rem:locdeg},
  we have to define moduli data as in Definition \ref{def:mod-data}. We will
  fix this according to the situation in algebraic geometry as follows. Assume
  first that $V$ lies on a vertex of $X$, so that $ X_V \cong L^3_2 $ after an
  integer linear isomorphism. Let $ \Sigma = (\sum_{i=0}^3\alpha_i^je_i)
  _{j=1,\dots,n} $ with $ \alpha_i^j \ge 0 $ be the degree of $ \Sigma $, where
  $ n = \val(V) $.

  We then consider four planes $ H_0,H_1,H_2,H_3 $ in $ \PP^3 $ in general
  position and count rational algebraic stable maps $ (C,x_1,\dots,x_n,f) $
  relative to these planes with intersection profiles $
  (\alpha_i^j)_{j=1,\dots,n} $ at $ H_i $ for all $i$. Their (finite) number
  will be the weight that we assign to the vertex $V$. In more complicated
  cases when this number is infinite, we expect that the corresponding
  (virtual) relative Gromov-Witten invariant would be the correct choice here,
  but for our situation at hand this problem does not occur.

  If $V$ lies on an edge of $X$, we assign a weight to $V$ analogously after
  projecting $ X_V \cong L^2_1 \times \RR $ to $ L^2_1 $.

  Here are two examples:
  \begin{enumerate}
  \item \label{con:surface0:a}
    For type (A) in the table above the only possible degree is $ \Sigma =
    (e_0+e_1,e_2+e_3) $ up to permutations, corresponding to lines in $ \PP^3
    $ passing to the two points $ H_0 \cap H_1 $ and $ H_2 \cap H_3 $. As
    there is exactly one such line, we assign to (A) the weight $1$.
  \item \label{con:surface0:b}
    For type (G) there are several possible degrees; as an example we will
    consider $ \Sigma = (3e_0+2e_1,e_1+e_2,e_2+e_3,e_2+2e_3) $, and thus count
    maps with $ f^* H_0 = 3x_1 $, $ f^* H_1 = 2x_1+x_2 $, $ f^* H_2 =
    x_2+x_3+x_4 $, $ f^* H_3 = x_3+2x_4 $ (in the Chow groups of the
    corresponding $ f^{-1}(H_i) $). Such a map would have to send $ x_1 $ to $
    H_0 \cap H_1 $, and $ x_3 $ and $ x_4 $ to $ H_2 \cap H_3 $. As $ f^* H_0 $
    only contains $ x_1 $ and $ f^* H_3 $ only contains $ x_3 $ and $ x_4 $,
    the curve would have to lie completely over the line through those two
    points. But then $ x_2 $ would have to map to both of these points
    simultaneously, which is impossible. Hence we assign the weight $0$ to this
    type.
  \end{enumerate}
\end{construction}

\begin{remark}[Conditions in resolution dimension 1] \label{rem:surface1}
  As the next step, we have to verify that, with the given moduli data, all
  vertices of resolution dimension $1$ are good. To show the general procedure
  will sketch this here for type (F), more details and the other cases can be
  found in \cite[Section 3.3]{Och13}.

  Note that the rays in this type must satisfy exactly the same linear relation
  as the rays of a tropical line. Also, none of the three possible resolutions
  is allowed to have a bounded edge of higher weight, as this does not occur
  for lines. It is checked immediately that this leaves only the degree
  $ \Sigma = (2e_0+e_1,e_1+e_3,e_2,e_2+e_3) $, up to isomorphism. It is shown
  in the picture below, together with its three resolutions in $X$.

  In order to embed the local moduli space into $\calM_{0,4}\times\R^3$, we
  evaluate the position of $x_3$ in $ \R^2 \cong \R^3/\langle e_2 \rangle $
  with coordinate directions $ e_1 $ and $ e_3 $, and the position of $ x_2 $
  in $ \R \cong \R^3/\langle e_1,e_3 \rangle $ with coordinate direction $e_2$
  (see Remark \ref{rem-m0nird}). Then the rays of the local moduli space are
  spanned by the vectors $ v_{\{1,3\}}+e_1-e_3 $, $ v_{\{1,4\}}+e_3 $, and $
  v_{\{1,2\}}-e_1 $ (in the notation of Construction \ref{constr-m0n}) for the
  three resolutions below, respectively. This is balanced with weights one,
  which are actually the gluing weights.

  \begin {center} \input {pics/vertresol1} \end {center}

  \begin {center} \begin{picture}(0,0)%
\includegraphics{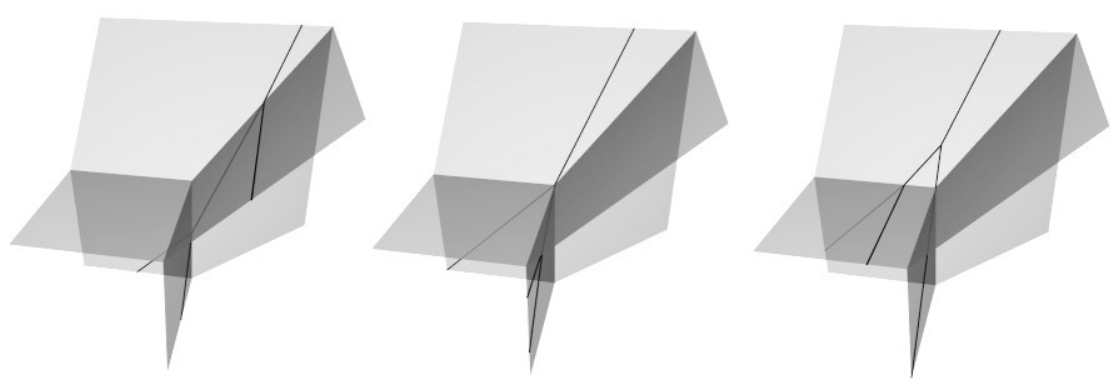}%
\end{picture}%
\setlength{\unitlength}{4144sp}%
\begingroup\makeatletter\ifx\SetFigFont\undefined%
\gdef\SetFigFont#1#2#3#4#5{%
  \reset@font\fontsize{#1}{#2pt}%
  \fontfamily{#3}\fontseries{#4}\fontshape{#5}%
  \selectfont}%
\fi\endgroup%
\begin{picture}(5109,1777)(676,-3683)
\end{picture}%
 \end {center}
\end{remark}

From Corollaries \ref{cor-balancing} and \ref{cor:good1dim} we therefore
conclude:

\begin{corollary} \label{cor-dimlines}
  With the moduli data of Construction \ref{con:surface0}, the moduli space
  $ \mm X1 $ of lines in a tropical surface $ X \subset \RR^3 $ is a tropical
  variety of dimension
    \[ \dim \mm X1 = 3-\deg X \]
  (and empty if $ \deg X > 3 $). In particular, this moduli space consists of
  finitely many (weighted) points if $\deg X=3$.
\end{corollary}

\begin{example}[Infinitely many lines in a tropical cubic surface \cite{Vig10}]
    \label{ex-infinite}
  Consider a floor-decomposed generic cubic surface where the
  three walls (represented by a line, a conic and a cubic) have the following
  relative position to each other, projected in the $e_3$-direction:

  \begin{center} \begin{tikzpicture}
    \draw (0,1)--(4,1);
    \draw (0,1.5)--(4.25,1.5);
    \draw (0,2.25)--(4.75,2.25);
    \draw (4,0)--(4,1);
    \draw (5,0)--(5,1.25);
    \draw (6.25,0)--(6.25,1.75);
    \draw (4.75,2.25)--(6.5,4);
    \draw (5.5,2)--(7,3.5);
    \draw (6.25,1.75)--(7.5,3);
    \draw (4.25,1.25)--(4.25,1.5);
    \draw (4.25,1.5)--(4.75,2);
    \draw (4.75,2)--(5.5,2);
    \draw (5.5,2)--(5.5,1.75);
    \draw (5.5,1.75)--(5,1.25);
    \draw (5,1.25)--(4.25,1.25);
    \draw (4,1)--(4.25,1.25);
    \draw (4.75,2)--(4.75,2.25);
    \draw (5.5,1.75)--(6.25,1.75);

    \begin{scope}[gray]
      \draw (2,0)--(2,2.75);
      \draw (3,0)--(3,3);
      \draw (0,2.75)--(2,2.75);
      \draw (0,3.5)--(2.25,3.5);
      \draw (2.25,3.5)--(3.875,5.125);
      \draw (3,3)--(4.5,4.5);
      \draw (2.25,3)--(3,3);
      \draw (2.25,3)--(2.25,3.5);
      \draw (2.25,3)--(2,2.75);
    \end{scope}

    \begin{scope}[lightgray]
      \draw (1,4.5)--(0,4.5);
      \draw (1,4.5)--(1,0);
      \draw (1,4.5)--(2,5.5);
    \end{scope}

    \fill (3,2.25) circle (2pt) node[above right] {$V$};
  \end{tikzpicture} \end{center}

  Such a cubic surface contains exactly 27 isolated lines that count with
  multiplicity 1. In addition, it has a $1$-dimensional family of lines not
  containing any of the others. All lines in this family have a vertex mapping
  to the point $V$ shown above, while the rest of them is mapped to maximal
  cells of $X$. General lines in this family are therefore decorated as in the
  following picture.

  \begin{center} \begin{tikzpicture}
    \draw (1.5,1)--(1.1,1.4);
    \draw (1.5,1)--(1.1,0.6);
    \draw (1.5,1)--(2.4,1);
    \draw (2.8,1.4)--(2.4,1);
    \draw (2.8,0.6)--(2.4,1);
    \fill (1.5,1) circle (2pt) node[below right=-1pt] {$V$};
  \end{tikzpicture} \end{center}

  If the vertex mapped to $V$ is $3$-valent, it has to be of resolution
  dimension $-1$. Therefore the only admissible line in this family, \ie the
  only line in the family actually in $\mm X1 $, is the one with no bounded
  edge, \ie with a $4$-valent vertex mapping to $V$. This vertex is then of
  resolution dimension $0$. To determine its type, we map the four rays of the
  cubic at $V$, which are $ -e_1+e_3 $, $ e_1-2e_3 $, $ -e_2-2e_3 $, and $
  e_2+3e_3 $, by an integer linear isomorphism to the four unit vectors.
  This maps the rays of the line to the degree $ \Sigma = (3e_0+2e_1,
  e_1+e_2,e_2+e_3,e_2+2e_3) $. As this line has weight $0$ by Construction
  \ref{con:surface0} \refx{con:surface0:b}, we conclude that the whole family
  does not contribute to the virtual number, and the degree of the $0$-cycle $
  \mm X1 $ is again $27$.
\end{example}

We therefore conjecture:

\begin{conjecture}
  For every smooth cubic surface $ X \subset \RR^3 $ the $0$-cycle $ \mm X1 $
  has degree $27$.
\end{conjecture}

\begin{appendix}

\section{Pulling back the diagonal of a smooth variety} \label{sec:diagonal}

Let $X$ be a partially open tropical cycle, and let $Y$ be a smooth tropical
variety. In order to glue moduli spaces in Section \ref{sec:gluing} we need the
pull-back of the diagonal $ \Delta_Y $ of $Y$ by some morphism $ f: X \to Y
\times Y $. But although the diagonal is locally a product of Cartier divisors
in this case, tropical intersection theory unfortunately does not yet provide a
well defined pull-back for it. This appendix therefore contains the technical
details necessary to construct a well-defined pull-back cycle $ f^*\Delta_{Y}
$.

First we brief\/ly review some facts about matroids and matroid fans from
\cite{FR10}. Let $M$ be a matroid of rank $r$ on a finite ground set $E$. To
every flat $F$ of $M$ we associate a vector $ e_F := \sum_{i\in F} e_i \in \R^E
$, where the $e_i$ are the negative standard basis vectors. Correspondingly, to
every chain of flats $ \emptyset \subsetneq F_1\subsetneq\cdots\subsetneq F_s =
E $ we assign a cone, spanned by $e_{F_1},\dots,e_{F_s}$, and $ -e_{F_s} $. Let
$\BB(M)$ denote the collection of all these cones, where the maximal ones are
equipped with weight $1$. This simplicial fan defines a tropical variety $
\BB(M) $ whose dimension is the rank of $M$. We call it the \emph{matroid fan}
associated to $M$.

Of special interest to us is the \textit{uniform matroid} $ U_{r,k} $ on a
ground set $E$ of cardinality $k$, with rank function $r(A)= \min(|A|,r) $.
Its associated matroid fan is $\BB(U_{r,k})\cong L^{k-1}_{r-1} \times \R$.

\begin{construction} \label{con:ratfuncs}
  By \cite[Section 4]{FR10} the diagonal $ \Delta_{\BB(M)}$ in
  $\BB(M)\times\BB(M)$ can be cut out by a product of rational functions: we
  have $ \Delta_{\BB(M)}=\varphi_1 \cdot\,\cdots\,\cdot \varphi_r \cdot
  (\BB(M)\times\BB(M)) $ for the rational functions $ \phi_i $ linear on the
  cones of $ \BB(M) $ determined by
    \[ \varphi_i(e_A,e_B)=\left\{ \begin{array}{cl}
       -1 & \mbox{if } r_M(A)+r_M(B)-r_M(A\cup B)\geq i, \\
       0 & \mbox{else}
       \end{array}\right. \]
  for flats $A,B$ of $M$, where $ r_M $ is the rank function of $M$. Moreover,
  recursively intersecting with the $\varphi_i$ yields a matroid fan in each
  intermediate step, hence a locally irreducible tropical variety;
  this will be important in the construction. If we want to specify the matroid
  $M$ in the notation, we will write $ \phi_i $ also as $ \phi_i^M $.
\end{construction}

We can now give the construction to pull back the diagonal from a smooth
tropical fan.

\begin{construction} \label{con:diagonal}
  Consider a morphism $ f: X \to Y\times Y$ where $ Y\cong L^{k-1}_{r-1} \times
  \R^m $. Then there is a (non-canonical) isomorphism $ \theta: Y\times\R \to
  \BB(M)\times \R^{m} $, where $M = U_{r,k}$, and $\theta$ maps the additional
  factor $\R$ onto the lineality space of the matroid fan. Associated to $f$ we
  denote by $ \tilde f $ the composition map
    \[ X \times \RR^2
       \;\;\stackrel{f \times \id}{\longrightarrow}\;\;
       Y \times Y \times \RR^2
       \;\;\cong\;\;
       (Y \times \RR) \times (Y \times \RR)
       \;\;\stackrel{\theta \times \theta}{\longrightarrow}\;\;
       \BB(M)^2 \times (\RR^m)^2. \]
  Let $ \psi_1,\dots,\psi_m $ denote functions which cut out the diagonal
  $\Delta_{\R^m}$, and consider the cocycle
    \[ \Phi_Y := \varphi_1 \cdot\,\cdots\,\cdot \varphi_r \cdot
                 \psi_1 \cdot\,\cdots\,\cdot \psi_m
       \qquad \text{on } \BB(M)^2 \times (\R^m)^2, \]
  where $\varphi_i$ are the functions on $ B(M)^2 $ from Construction
  \ref{con:ratfuncs} above. One verifies immediately that the pull-back $
  \tilde f^* \Phi_Y \cdot (X \times \R^2) $ has the lineality space $ L_X :=
  0\times\Delta_\R $ in $ X \times \RR^2 $. So we can take the quotient by $
  L_X $ and use the projection $ p_X: (X \times \R^2)/L \to X $ to define
    \[ f^*\Delta_Y := f^* \Delta_Y \cdot X :=
       {p_X}_* \left[ (\tilde f^* \Phi_Y \cdot (X \times \R^2)) / L_X \right].
       \]
  As the intermediate steps in Construction \ref{con:ratfuncs} are locally
  irreducible, it follows from \cite[Lemma 3.8.13]{Fra12} that the support of $
  \tilde f^* \Phi_Y \cdot (X \times \RR^2) $ lies over the diagonal of $ \RR $
  in $ \RR^2 $ (so that $ p_X $ is injective on $ (\tilde f^* \Phi_Y \cdot (X
  \times \R^2)) / L_X $, in accordance with our convention in Remark
  \ref{rem-push-pull}), and that the support of $ f^* \Delta_Y $ lies in $
  f^{-1}(\Delta_Y) $.

  Moreover, it can be shown that this definition depends neither on the choice
  of $ \psi_1,\dots,\psi_m $ \cite[Theorem 2.25]{Fra11} nor on the choice of
  isomorphism $ \theta $ \cite[Lemma 1.4.3]{Och13}. However, it is not known
  whether it depends on the choice of the rational functions $
  \phi_1,\dots,\phi_r $ cutting out the diagonal of $ \BB(M) $.
\end{construction}

Let us brief\/ly state the main properties of this definition that follow from
the compatibilities between the various intersection-theoretic constructions.

\begin{lemma} \label{lem:diag-compat} ~

  \vspace{-1ex}

  \begin{enumerate}
  \item (Projection formula) \label{lem:diag-compat:projform}
    For two morphisms $ Z\stackrel{g}{\longrightarrow} X
    \stackrel{f}{\longrightarrow}Y\times Y$, where $g$ is injective and $Y$ a
    smooth fan, we have
      \[ g_*\left[(f \circ g)^* \Delta_{Y}\cdot Z\right]
         = f^*\Delta_{Y}\cdot g_* Z. \]
  \item (Quotients) \label{lem:diag-compat:quotients}
    Let $X$ be a partially open tropical variety with lineality space $L$ and
    quotient map $ q: X \to X/L $, and let $ f: X/L \to Y \times Y $ be a
    morphism for a smooth fan $Y$. Then
      \[ q \left( (f \circ q)^* \Delta_Y \cdot X \right)
         = f^* \Delta_Y \cdot (X/L). \]
  \item (Commutativity) \label{lem:diag-compat:comm}
    For two morphisms $ f: X \to Y\times Y$ and $ g: X \to Z\times Z$ to smooth
    fans $Y$ and $Z$ we have
      \[ f^*\Delta_{Y} \cdot \left(g^*\Delta_{Z}\cdot X\right) =
         g^*\Delta_{Z} \cdot \left(f^*\Delta_{Y}\cdot X\right). \]
  \item (Projections) \label{lem:proj}
    Let $X$ and $Z$ be partially open tropical varieties, and denote by $ p: X
    \times Z \to X $ the projection. For any morphism $ f: X \to Y \times Y $
    for a smooth fan $Y$ we have
      \[ (f \circ p)^* \Delta_Y \cdot (X \times Z) = (f^* \Delta_Y \cdot X)
         \times Z. \]
  \item (Products) \label{lem:diag-compat:prod}
    Let $ f: X \to Y \times Y $ and $ f': X' \to \RR^k \times \RR^k $ be two
    morphisms, for a smooth fan $Y$. Then
      \[ (f \times f')^* \Delta_{Y \times \RR^k} \cdot (X \times X')
         = (f^* \Delta_Y \cdot X) \times (f'{}^* \Delta_{\RR^k} \cdot X'). \]
  \end{enumerate}
\end{lemma}

\begin{proof}
  All these statements can be checked immediately, see
  \cite[Section 1.4]{Och13} for details. As an example, we show part (a): We
  have
  \begin{alignat*} 2
    g_* \left[(f \circ g)^* \Delta_Y \cdot Z \right] 
    &= g_* {p_Z}_* \left[ \big( {\widetilde{f\circ g}}^* \Phi_{Y}
       \cdot(Z\times\R^2) \big) /L_Z \right] \\
    &= {p_X}_* (g\times\id)_* \left[ \big( (g \times \id)^* \tilde f^* \Phi_{Y}
       \cdot(Z\times\R^2) \big) /L_Z \right]
       && \quad \text{(functoriality)} \\
    &= {p_X}_* \left[ \big( (g\times\id)_* (g \times \id)^* \tilde f^* \Phi_{Y}
       \cdot(Z\times\R^2) \big) /L_X \right]
       && \quad \text{(Lemma \ref{lem:quot-push})} \\
    &= {p_X}_* \left[\big( \tilde f^* \Phi_{Y} \cdot(g_*(Z)\times\R^2) \big)
       /L_X \right]
       && \quad \text{(projection formula)} \\
    &= f^*\Delta_{Y} \cdot g_*(Z),
  \end{alignat*}
  where we have used the projection formula for cocycles as in
  \cite[Proposition 2.24 (3)]{Fra11}.
\end{proof}

So far $ Y \cong L^{k-1}_{r-1} \times \RR^m $ was assumed to be a smooth fan.
In order to generalize this to smooth varieties we need the following
compatibility statement. Let $ \sigma $ be a cell of $Y$ in the coarsest
subdivision, with relative interior $ \sigma^\circ $. We consider both $ \sigma
$ and $ \sigma^\circ $ as subsets of the diagonal $ \Delta_Y \subset Y \times Y
$. Let $ Y(\sigma) $ be the union of all open cells in the matroid subdivision
of $ Y \times Y $ whose closure intersects $ \sigma^\circ $, which is then an
open neighborhood of $ \sigma^\circ $ in $ Y \times Y $. It is also contained
in $ Y_\sigma \times Y_\sigma $, where $ Y_\sigma $ is the star of $Y$ at $
\sigma $. As $ Y_\sigma $ is again a smooth fan, we can regard the restriction
of $f$ to $ X_\sigma := f^{-1}(Y(\sigma)) $ also as a morphism $ f_\sigma:
X_\sigma \to Y_\sigma \times Y_\sigma $, and apply Construction
\ref{con:diagonal} to this map. As expected, we will now show that over the
cell $ \sigma $ this gives the same result as for $ f: X \to Y \times Y $.

\begin{lemma}[Compatibility] \label{lem:compat}
  Let $ f : X \to Y \times Y $ be a morphism, with $Y$ a smooth fan. Moreover,
  let $ \sigma $ be a cell in the coarsest subdivision of $Y$. With notations
  as above, the weights of $ f^* \Delta_Y \cdot X $ and $ f_\sigma^*
  \Delta_{Y_\sigma} \cdot X_\sigma $ then agree on all cells of $X$ whose
  interior is mapped by $f$ to the cell $ \sigma^\circ $ in the diagonal $
  \Delta_Y $.
\end{lemma}

\begin{proof}
  Let $ Y \cong L^{k-1}_{r-1} \times \RR^m $, and let $ M=U_{r,k} $ be the
  corresponding uniform matroid on $ E=\{1,\dots,k\} $, so that $ Y \times \RR
  \cong B(M) \times \RR^m $. The cell $ \sigma $ then corresponds to a subset $
  S \subset E $, i.\,e.\ it consists of all cells in the matroid subdivision
  for chains of flats in $S$. If $ \dim \sigma = s $ then $ Y_\sigma \cong
  L^{k-s-1}_{r-s-1} \times \RR^{m+s} $, or more precisely $ Y_\sigma \times \RR
  \cong B(M_S) \times \RR^s \times \RR^m $, where $ M_S $ is the uniform
  matroid of rank $ r-s $ on $ E \backslash S $.

  As Construction \ref{con:diagonal} is local, it suffices to show that the
  rational functions cutting out the diagonal in this construction are the same
  for the spaces $ Y \times Y \times \RR^2 \cong B(M)^2 \times (\RR^m)^2 $ and
  $ Y_\sigma \times Y_\sigma \times \RR^2 \cong B(M_S)^2 \times (\RR^s)^2
  \times (\RR^m)^2 $ when restricted to the common open neighborhood
  $ Y(\sigma) \times \RR^2 $ of the cell $ \sigma^\circ \times \RR $ in the
  diagonal. For this we need to prove that these rational functions agree on
  all rays of $ \overline{Y(\sigma)} $. As $ (e_S,e_S) $ is the only interior
  ray of $ \sigma^\circ \times \RR $, the maximal cells of $
  \overline{Y(\sigma)} $ correspond to maximal chains of flats in $ M \oplus M
  $ containing $ (S,S) $, and hence we have to compare the rational functions
  of Constructions \ref{con:ratfuncs} and \ref{con:diagonal} on all rays $
  (e_A,e_B) $ for flats $A,B$ of $M$ with $ A,B \subset S $ or $ A,B \supset S
  $.

  The functions cutting out the diagonal of $ \RR^m $ can obviously be chosen
  to be the same in both cases. By Construction \ref{con:ratfuncs}, the others
  are $ \phi_1^M,\dots,\phi_k^M $ for $ B(M)^2 $, and $ \phi_1^{M_S},\dots,
  \phi_{k-s}^{M_S} $ and $ \phi_1^S,\dots,\phi_s^S $ for $ B(M_S)^2 \times
  (\RR^s)^2 $ (where $S$ stands for the uniform matroid of full rank on $S$, so
  that $ \phi_1^S,\dots,\phi_s^S $ can be used to cut out the diagonal of $
  \RR^s $). Now, on the rays $ (e_A,e_B) $ mentioned above \dots
  \begin{itemize}
  \item $ \phi_i^M $ agrees with $ \phi_{i-s}^{M_S} $ for $ i=s+1,\dots,k $:

    If $ A,B \subset S $ then both $ r_M(A) + r_M(B) - r_M(A \cup B)
    \ge i $ and $ r_{M_S}(A \backslash S) + r_{M_S}(B \backslash S) -
    r_{M_S} ((A \cup B) \backslash S) \ge i-s $ are never satisfied.

    If $ A,B \supset S $ then $ r_M(A) = r_{M_S}(A \backslash S) + s $,
    and similarly for $B$ and $ A \cup B $. Hence $ r_M(A) + r_M(B) -
    r_M(A \cup B) \ge i $ is equivalent to $ r_{M_S}(A \backslash S) +
    r_{M_S}(B \backslash S) - r_{M_S} ((A \cup B) \backslash S) \ge i-s $.
  \item $ \phi_i^M $ agrees with $ \phi_i^S $ for $ i=1,\dots,s $:

    If $ A,B \subset S $ then $ r_M(A) + r_M(B) - r_M(A \cup B) \ge i $ is
    equivalent to $ r_S(A \cap S) + r_S(B \cap S) - r_S((A \cup B) \cap S)
    \ge i $.

    If $ A,B \supset S $ then both $ r_M(A) + r_M(B) - r_M(A \cup B) \ge i $
    and $ r_S(A \cap S) + r_S(B \cap S) - r_S((A \cup B) \cap S) \ge i $ are
    always satisfied.
  \end{itemize}
\end{proof}

\begin{remark}[Pullbacks of diagonals of smooth varieties]
    \label{rem:pull-smooth}
  Lemma \ref{lem:compat} implies that we cannot only pull back diagonals of
  smooth fans, but also of smooth varieties: Let $ f : X \to Y \times Y $ be a
  morphism from a partially open tropical cycle to a smooth tropical variety.

  To assign a weight to a cell $ \tau $ of dimension $ \dim X - \dim Y $ over
  the diagonal $ \Delta_Y $, let $ \sigma $ be the cell in $ Y \cong \Delta_Y
  $ so that the relative interior of $ \tau $ maps to the relative interior of
  $ \sigma $. Choose any face $ \sigma' $ of $ \sigma $ (which might be $
  \sigma $ itself), replace $Y$ by the star $ Y_{\sigma'} $ at this face and
  $X$ by the open subcycle of $X$ consisting of all points mapping to $ \sigma'
  $ or any of its adjacent open cells in both components of $ Y \times Y $, and
  assign to $ \tau $ its weight in the cycle $ f_{\sigma'}^*
  \Delta_{Y_{\sigma'}} \cdot X_{\sigma'} $. By Lemma \ref{lem:compat} the
  result does not depend on the choice of $ \sigma' $.
  
  Using the same local procedure for a cell $ \tau $ of dimension $ \dim X -
  \dim Y - 1 $, we obtain a balanced cycle $ f_{\sigma'}^* \Delta_{Y_{\sigma'}}
  \cdot X_{\sigma'} $ including all adjacent cells of dimension $ \dim X - \dim
  Y $. Hence our local construction glues to give a well-defined cycle $ f^*
  \Delta_Y \cdot X $.
\end{remark}

\end{appendix}

\bibliographystyle{amsalpha}
\bibliography{bibliographie}

\end{document}